\documentclass[11pt, reqno]{amsart}

\usepackage{amssymb,amsmath,amsthm,amsfonts,tikz, hyperref, amscd,amstext}
\usepackage{xypic}
\usepackage{array}
\usepackage[all]{xy}

\newtheorem{thm}{Theorem}[section]
\newtheorem{lemma}[thm]{Lemma}
\newtheorem{prop}[thm]{Proposition}

\newtheorem{cor}[thm]{Corollary}

\theoremstyle{definition}
\newtheorem{definition}[thm]{Definition}
\newtheorem{remark}[thm]{Remark}

\newtheorem{ntt}[thm]{Notation}

\newtheorem{example}[thm]{Example}

\newcommand{\Ad}{\operatorname{Ad}}
\newcommand{\Aut}{\operatorname{Aut}}
\newcommand{\SL}{\operatorname{SL}}
\newcommand{\GL}{\operatorname{GL}}

\newcommand{\U}{\operatorname{U}}

\newcommand{\N}{\mathbb{N}}
\newcommand{\Z}{\mathbb{Z}}
\newcommand{\Q}{\mathbb{Q}}
\newcommand{\R}{\mathbb{R}}

\newcommand{\T}{\mathbb{T}}
\newcommand{\TT}{\mathcal{T}}

\newcommand{\A}{\mathcal{A}}
\newcommand{\BB}{\mathcal{B}}

\newcommand{\af}{\alpha}
\newcommand{\bt}{\beta}

\newcommand{\Ld}{\Lambda}
\newcommand{\w}{\wedge}
\newcommand{\rk}{\operatorname{rk}}

\begin{document}

\title{Finite groups acting on higher dimensional noncommutative tori}

\author[J. A. Jeong]{Ja A Jeong$^{\dagger}$}
\thanks{Research partially supported by  NRF-2012R1A1A2008160$^{\dagger}$.}
\address{
Department of Mathematical Sciences and Research Institute of Mathematics\\
Seoul National University\\
Seoul, 151--747\\
Korea} \email{jajeong\-@\-snu.\-ac.\-kr }

\author[J. H. Lee]{Jae Hyup Lee $^{\dagger}$}
\address{
Department of Mathematical Sciences\\
Seoul National University\\
Seoul, 151--747\\
Korea} \email{jjub9831\-@\-snu.\-ac.\-kr}



\begin{abstract}
For the canonical action $\alpha$ of $\operatorname{SL}_2(\mathbb{Z})$ on 
2-dimensional simple rotation algebras $\mathcal{A}_\theta$,  
it is  known that if $F$ is a finite subgroup of $\operatorname{SL}_2(\mathbb{Z})$, 
the crossed products $\mathcal{A}_\theta\rtimes_\alpha F$ are all AF algebras. 
  In this paper we show that  this is not the case for higher dimensional noncommutative tori. 
  More precisely, we show that 
  for each $n\geq 3$ there exist noncommutative simple $\phi(n)$-dimensional tori 
  $\mathcal{A}_\Theta$  which admit canonical action of $\mathbb{Z}_n$ and for each
  odd $n\geq 7$ with $2\phi(n)\geq n+5$ their crossed products 
  $\mathcal{A}_\Theta\rtimes_\alpha \mathbb{Z}_n$ are not AF 
  (with nonzero $K_1$-groups).  
It is also shown that the only possible canonical action by a finite group on a 
$3$-dimensional simple torus is the flip action by $\mathbb{Z}_2$. 
Besides, we discuss the canonical actions by finite groups 
$\mathbb{Z}_5, \mathbb{Z}_8, \mathbb{Z}_{10}$, and $\mathbb{Z}_{12}$ 
on the $4$-dimensional torus 
of the form $\mathcal{A}_\theta\otimes \mathcal{A}_\theta$. 
\end{abstract}

\maketitle

\section{Introduction}

\vskip 1pc

\noindent 
The {\it rotation algebra} $\A_\theta$, $\theta \in \R$, 
is the universal $C^*$-algebra generated by two unitaries $u_1, u_2$ satisfying 
the commutation relation
$u_2 u_1=\exp (2 \pi i \theta)u_1 u_2$. 
If $u_1$ and $u_2$ commute (that is, if $\theta \in \Z$), 
$\A_\theta$ is isomorphic to the commutative 
$C^*$-algebra $C(\T^2)$ of all continuous functions 
on the $2$-dimensional torus $\T^2$, and so the rotation algebras  
$\A_\theta$ are often called  $2$-dimensional noncommutative tori.
If $\theta\in \R\setminus \Q$, $\A_\theta$ is called an 
{\it irrational rotational algebra} and this is the case exactly when 
$\A_\theta$ is a simple $C^*$-algebra. 

\vskip .5pc

More generally, for $d\geq 2$, 
a {\it noncommutative $d$-dimensional torus} (or simply 
a {\it  $d$-torus}) $\A_\Theta$ associated with a skew symmetric real 
$d\times d$ matrix $\Theta=(\theta_{kj})$ is  
the universal $C^*$-algebra generated by $d$ unitaries 
$u_1,\dots,u_d$ that are subject to the commutation relations
\begin{equation}\label{commutation relation}
 u_j u_k = \exp(2\pi i \theta_{kj}) u_k u_j.
\end{equation} 
 $\A_\Theta$  was introduced in \cite{Rf90} as 
the twisted group algebra $C^*(\Z^d, \omega_\Theta)$  
of $\Z^d$ twisted by the 2-cocycle $\omega_\Theta$ given in (\ref{omega}). 

\vskip .5pc

In \cite{Wt} Watatani considered an automorphism $\af_A$,  
$A=( a_{ij}) \in \SL_2(\Z)$,  
on an irrational rotational algebra $\A_\theta$ defined by  
\begin{equation}\label{canonical action}   
   \af_A(u_1)= \exp(\pi i \theta a_{11}a_{21})u_1^{a_{11}} u_2^{a_{21}},\  
   \af_A(u_2)= \exp(\pi i \theta a_{12}a_{22})u_1^{a_{12}} u_2^{a_{22}} 
\end{equation}
and then classified these automorphisms using the notion of $K_1$-entropy. 
Brenken \cite{Br} used the automorphism to study representations of 
rotational algebras. 
In this paper, the action $A\mapsto \af_A: \SL_2(\Z)\to Aut(A_\theta)$ 
and its $d$-dimensional version (Definition~\ref{canonical automorphism}) 
will be called a {\it canonical action}. 
 
\vskip .5pc
 
The group $\SL_2(\Z)$ is known to have only four (up to conjugacy) 
nontrivial finite subgroups which are isomorphic to $\Z_2$, $\Z_3$, $\Z_4$, and $\Z_6$. 
 The crossed products $\A_\theta\rtimes_\af \Z_k$
 of a simple $\A_\theta$  
 by the restriction of the canonical action $\af$ to $\Z_k$, $k=2,3,4,6$, 
 are all known to be AF-algebras and moreover their $K_0$ groups 
 are computed (see \cite[Theorem 0.1]{ELPW}), which implies 
 $\A_{\theta_1}\rtimes_\af \Z_k\cong \A_{\theta_2}\rtimes_\af \Z_l$ 
 if and only if  $k=l$ and $\theta_1=\pm \theta_2 \  {\rm mod}\, \Z$.
Also it is known in the same paper \cite{ELPW} that 
$\A_\Theta\rtimes_\sigma \Z_2$ is an AF algebra if $\A_\Theta$ is a simple 
$d$-dimensional noncommutative torus and $\sigma$ is the 
action given by the flip automorphism sending the unitary 
generators $u_j$ to their adjoints $u_j^*$ for $j=1,\dots, d$. 
This seminal work \cite{ELPW} was actually motivated, 
as reviewed in the first chapter there, 
by several previous studies including, for example, 
the result \cite{Wl} that for most irrational  numbers  $\theta$, 
the crossed products $\A_\theta\rtimes_\af \Z_4$ are AF algebras,  
and it finally settled down the case $\A_\theta\rtimes_\af F$ 
for any (2-dimensional) irrational rotational algebras 
$\A_\theta$  and any finite groups $F\subset \SL_2(\Z)$.  

\vskip .5pc

It would then be a very natural question to ask whether the crossed product  
$\A_\Theta\rtimes_\af G$ of a  simple higher dimensional noncommutative $d$-torus 
$\A_\Theta$ is still AF even when $\af$ is the canonical action of a finite subgroup 
$G$ of $\SL_d(\Z)$ (or $\GL_d(\Z)$). 
But it was unclear, at least to the knowledge of 
the authors, even whether there are
any known finite groups acting canonically on 
some higher dimensional noncommutative simple tori, 
except the flip action by $\Z_2$,  when the authors got interested in this question. 
The purpose of this paper is thus to find finite subgroups $G$ of 
$\GL_d(\Z)$ which act canonically on 
higher dimensional noncommutative simple $d$-tori $\A_\Theta$ 
and to figure out if there are simple crossed products $\A_\Theta\rtimes_\af G$ 
which are not AF. 

\vskip .5pc

In order to show that such a crossed product is not AF,
it is enough to see that the $K_1$-group of $\A_\Theta\rtimes_\af G$ is nonzero.
From the general theory developed in \cite{ELPW}, we can 
deduce without difficulty that the $K$-groups $K_*(\A_\Theta\rtimes_\af G)$ 
are equal to the $K$-groups $K_*(C^*(\Z^d\rtimes G))$ of the  
semidirect product group $C^*$-algebras. 

\vskip .5pc

For the finite subgroups $G$ of $\GL_d(\Z)$, we will use the companion matrix 
$C_n$ of the $n$th cyclotomic polynomial; the matrix 
$C_n$ is a $d\times d$ matrix, $d=\phi(n)$, and is of order $n$. 
The finite cyclic group $\Z_n=\langle C_n \rangle$ generated by $C_n$ 
then acts on $\Z^d$ by conjugation and we have the semidirect 
product group $\Z^d\rtimes \Z_n$. 
We show  the following theorem with the aid of 
the recent results on topological $K$-theory of 
group $C^*$-algebras known in \cite{LL12}:

\vskip .5pc

\begin{thm}{\rm (Theorem~\ref{nonzero K_1})}\label{thm 1} 
Let $n\geq 7$ be an odd integer with $d:=\phi(n)$. 
If $2d\geq n+5$, then $$K_1(C^*(\Z^d\rtimes_\af \Z_n))\neq 0,$$ 
where $\af$ is the conjugation action of $\Z_n=\langle C_n \rangle$ on $\Z^d$.  
\end{thm} 
 
\vskip .5pc 
\noindent 
For a prime number $n\geq 3$, it is known in \cite[Theorem 0.3]{DL} that 
$K_1(C^*(\Z^d\rtimes_\af \Z_n))= 0$ if and only if $n=3$ or 5. 

\vskip .5pc 
The remaining thing to answer our question is then to show that 
the cyclic group $\Z_n=\langle C_n \rangle$ can really act canonically 
on some noncommutative simple $d$-dimensional tori $\A_\Theta$, and 
we prove the following:

\vskip .5pc 

\begin{thm}{\rm (Theorem~\ref{nondegenerate theta})}\label{thm 2} 
Let $n\geq 3$ and $d:=\phi(n)$. 
Then there exist simple $d$-dimensional tori $\A_\Theta$ on which 
the group $\Z_n=\langle C_n\rangle$ acts canonically.
\end{thm}

\vskip .5pc
\noindent
From this result and Theorem~\ref{thm 1}, we can say that 
there exist many noncommutative simple higher dimensional 
tori whose crossed products by finite cyclic groups via canonical actions are 
not AF. 

\vskip .5pc
If $p\geq 3$ is prime (thus $d =p-1$), 
we can say more on the skew symmetric $d\times d$ 
matrices $\Theta$, and apply \cite[Theorem 0.3]{DL} 
to figure out exactly when the simple crossed 
product $\A_\Theta\rtimes \Z_p$ is AF:  

\vskip .5pc

\begin{thm}{\rm (Theorem~\ref{action exists} and Corollary~\ref{cor for prime case})}\label{thm 3} 
Let $p\geq3$ be a prime. 
Then  noncommutative $d$-tori $\A_\Theta$ associated with $\Theta$ 
of the form in (\ref{prime theta form})  
admit the canonical action of $\Z_p=\langle C_p \rangle$. Conversely, 
if  $\Z_p=\langle C_p \rangle$ canonically acts on a $d$-torus $\A_\Theta$, 
then $\Theta$ must have the form in  (\ref{prime theta form}).
For the crossed product $\A_\Theta\rtimes_\af \Z_p$ of a simple $\A_\Theta$, 
it is an AF algebra if and only if $p=3$ or $5$. 
\end{thm}  
 
\vskip .5pc 

Since the cases considered above do not include odd-dimensional 
noncommutative tori  
while $3$-dimensional case, for example, is expected to be easily understood and maybe 
much similar to 2-dimensional tori than tori with dimension a lot higher, 
we examine noncommutative simple 3-tori separately 
and obtain the following: 

\vskip 1pc

\begin{thm}{\rm (Theorem~\ref{3-torus})}\label{thm 4} 
The only canonical action by a nontrivial finite cyclic group 
on a simple $3$-dimensional torus is the flip action by $\Z_2$.
\end{thm}  

\vskip .5pc
 
Going one step further 
we will also  examine the canonical actions by finite groups on 4-dimensional tori. 
We know from Theorem~\ref{thm 3} with $p=5$  that 
there are simple 4-tori admitting canonical actions by $\Z_5=\langle C_5\rangle$ 
and their crossed products are all AF. 
  Among 4-dimensional simple tori, 
  especially we will look at $\A_\Theta$ isomorphic to  
the tensor product $\A_\theta\otimes \A_\theta$ 
of a simple 2-dimensional torus $\A_\theta$ with itself and 
describe explicitly in Proposition~\ref{4-torus} the canonical actions by the groups 
$\Z_n= \langle C_n\rangle$ for $n=5,8,10,12$.

\vskip .5pc

This paper is organized as follows. 
In Section 2, we first review  definitions and important results  which
we need in later chapters and then make it clear why  
the resulting crossed product  under consideration in this paper 
is AF exactly when its $K_1$ group is zero.  
 In Section 3 we explain how the machinery obtained in \cite{ELPW} 
 and the results in \cite{DL, LL12} can be applied to our 
 situation, and then prove Theorem~\ref{thm 1}. 
In Section 4, we prove Theorem~\ref{thm 2} and Theorem~\ref{thm 3}.
 Then applying the  complete list of elements in $\GL_3(\Z)$ of finite order 
 known in \cite{Th}, we prove Theorem~\ref{thm 4}  in Section 5. 
 Finally in Section 6,  we 
 examine in Proposition~\ref{4-torus} some of the $4$-dimensional simple tori  and 
 the canonical actions by the finite groups $\Z_5, \Z_8, \Z_{10}, \Z_{12}$.

\vskip 1pc

{\it Acknowledgement.} The authors have benefited from useful suggestions 
by N. C. Phillips. 

\vskip 1pc

\section{Preliminaries}

\noindent
In this section we recall basic definitions and important facts 
(from \cite{DL, ELPW, LL12, Lin05,  PR, Rf88, Rf90, RS}) 
which shall be used throughout the paper. 

\subsection{Twisted group algebras of discrete groups} 
Let $G$ be a second countable discrete group. A 2-cocycle on $G$ is 
a function $\omega: G\times G\to \mathbb T$ such that 
$\omega(x,y)\omega(xy,z)=\omega(y,z)\omega(x, yz)$ 
and $\omega(x,1)=\omega(1,x)=1$ for $x,y,z\in G$. 
By $\ell^1(G,\omega)$ we denote the twisted convolution $*$-algebra of all 
summable functions on $G$ with   
\begin{align*}
(f*_\omega g)(x)& =\sum_{y\in G} f(y)g(y^{-1}x)\omega(y,y^{-1}x)\\
 f^*(x)&=\overline{\omega(x,x^{-1})f(x^{-1})}.
\end{align*} 
We call a map $v: G\to \mathcal U(\mathcal H)$ of $G$ into 
the unitary group of a Hilbert space $\mathcal H$  
an $\omega$-{\it representation} of $G$ if 
\begin{eqnarray}\label{omega representation} 
v(x)v(y)=\omega(x,y)v(xy)
\end{eqnarray} 
for $x,y\in G$. 
The {\it regular} $\omega$-{\it representation} of $G$ is 
the $\omega$-representation $l_\omega:G\to \mathcal U(\ell^2(G))$ given by 
$$ (\l_\omega(x)\xi)(y)=\omega(x, x^{-1}y)\xi(x^{-1}y)$$
for $\xi\in \ell^2(G)$ and $x,y\in G$.
 Every $\omega$-representation $v: G\to \mathcal U(\mathcal H)$  
  induces a contractive $*$-homomorphism  
 $v:\ell^1(G,\omega)\to B(\mathcal H)$ (also denoted $v$) 
 given by  
 $v(f)=\sum_x f(x)v(x)$ for $f\in \ell^1(G,\omega)$, 
 and every nondegenerate representation of $\ell^1(G,\omega)$ 
 arises in this way.  
The {\it full twisted group algebra} $C^*(G,\omega)$ is then   
defined to be the enveloping $C^*$-algebra of $\ell^1(G,\omega)$ and 
the {\it reduced twisted group algebra} $C^*_r(G,\omega)$ is 
the image of $C^*(G,\omega)$ under the regular representation $l_\omega$.
 If $G$ is amenable, $C^*(G,\omega)$ is equal to 
 $C^*_r(G,\omega)$ by \cite[Theorem 3.11]{PR},   
 and hence by (\ref{omega representation}) 
$$C^*(G,\omega)=\overline{\rm span}\{l_\omega(x):x\in G\}.$$ 
  
\vskip .5pc
\subsection{Noncommutative tori}

A real skew symmetric $d\times d$ matrix $\Theta$ induces a 2-cocycle 
$\omega_\Theta:\Z^d\times \Z^d\to \T$ 
given by 
\begin{equation}\label{omega}
\omega_\Theta(x,y)=\exp (\pi i \langle \Theta x,y\rangle) 
\end{equation} 
for $x,y\in \mathbb Z^d$. 
The twisted group algebra $C^*(\Z^d,\omega_\Theta)$ is called 
a {\it noncommutative $d$-torus} (\cite{Rf90}).
If  $\{e_i\}_{i=1}^d$ is the standard basis of $\mathbb Z^d$, 
then $\omega_\Theta(e_j,e_k) =\exp (\pi i \theta_{kj})$ and    
$C^*(\Z^d,\omega_\Theta)=C^*\{l_\Theta(e_i):i=1,\dots, d\}$, where 
$l_{\Theta}:\mathbb Z^d\to \mathcal U(\ell^2(\mathbb Z^d))$ denotes 
the regular $\omega_\Theta$-representation.
 It is also easy to see that for $x,y\in \Z^d$, 
 $\omega_\Theta(x,x)=1$ and     
$l_\Theta (x)l_\Theta(y)=\omega_\Theta(x,y)  l_\Theta (x+y) 
=\omega_\Theta(x,y)  l_\Theta (y+x) 
=\omega_\Theta(x,y) \overline{\omega_\Theta(y,x)} l_\Theta (y)l_\Theta(x) 
=\omega_\Theta(x,y)^2 l_\Theta (y)l_\Theta(x)$. Thus   
\begin{eqnarray}\label{omega commuting relation}
l_\Theta (e_j)l_\Theta(e_k)
=\omega_\Theta(e_j,e_k)^2 l_\Theta (e_k)l_\Theta(e_j) 
=\exp ({2\pi i \theta_{kj}})l_\Theta (e_k)l_\Theta(e_j)
\end{eqnarray}
follows for $j,k=1,\dots,d$, and 
\begin{equation}\label{expansion}
l_\Theta(y)=\exp\big(\pi i \sum^d_{k=2}
\sum^{k-1}_{j=1}y_k y_j\theta_{jk}\big)
 \,l_\Theta(e_1)^{y_1}\cdots l_\Theta(e_d)^{y_d}
\end{equation}
holds for $y=(y_1,\dots,y_d)\in \Z^d$.
The relation (\ref{omega commuting relation}) shows that 
the generating unitaries $\{l_\Theta (e_j)\}_{j=1}^d$ 
of $C^*(\Z^d,\omega_\Theta)$ satisfy the relation (\ref{commutation relation}).  
In fact  $C^*(\Z^d,\omega_\Theta)$ is characterized 
as the universal $C^*$-algebra generated by $d$ unitaries $\{u_j\}_{j=1}^d$ satisfying 
the relations  (\ref{commutation relation}) (\cite{Rf90}).  
 Also,  $C^*(\Z^d,\omega_\Theta)$ is usually denoted by $\A_\Theta$;  
\begin{eqnarray} \label{nc torus}
\A_\Theta=C^*(\Z^d,\omega_\Theta). 
\end{eqnarray} 
For $\theta\in \mathbb R$,  the rotation algebra  $\A_\theta$ 
is the noncommutative 2-torus $\A_\Theta$ associated to the 
real skew symmetric $2\times 2$ matrix $\Theta=(\theta_{kj})$ with $\theta_{12}=\theta$.
Of course,  $\A_\Theta$ is not necessarily noncommutative as   
the generators commute each other  
if $\theta_{kj}\in \mathbb Z$ for all $k,j=1,\dots, d$. 

\vskip 1pc

\begin{ntt}
As in \cite{RS}, we use the following notation: 
$$\TT_{d} (\R):=\{\,\Theta\in M_d(\R): \, \Theta^t=-\Theta \,\},$$
where $\Theta^t$ denotes the transpose of $\Theta$. 
Similarly, $\TT_{d} (\Z)$ denotes the set of all 
 $d\times d$ skew symmetric matrices with entries from $\Z$.
For a skew symmetric matrix $\Theta\in \TT_d(\R)$, we will consider the group  
$$G_{\Theta}:=\{A\in \GL_d(\Z): \Theta = A^t\Theta A  \,\}.$$
Actually it is the isotropy group of 
$\Theta\in \TT_d(\R)$ under the action of $\GL_d(\Z)$,  
$$(A, \Theta)\mapsto (A^{-1})^t\Theta A^{-1}:\GL_d(\Z)\times \TT_d(\R)\to \TT_d(\R).$$ 
\end{ntt}

\vskip 1pc

We  call $\Theta\in \TT_{d} (\R)$ {\it nondegenerate} 
if whenever  $x\in \mathbb Z^d$ satisfies  $\exp(2\pi i \langle x, \Theta y\rangle)=1$
for all $y\in \mathbb Z^d$, then $x=0$. 
Otherwise $\Theta$ is called {\it degenerate}.  
For the simplicity of the algebra $\A_\Theta$, the following is known. 

\vskip 1pc

\begin{thm}{\rm (\cite[Theorem 1.9]{Ph06}, \cite[Theorem 3.7]{Sl})} 
\label{simple}
Let $\Theta\in \TT_{d} (\R)$. Then the noncommutative $d$-torus 
$\A_\Theta$ is simple if and only if $\Theta$ is nondegenerate.
\end{thm}

\vskip 1pc

\subsection{Canonical action on noncommutative tori $\A_\Theta$}  
 
\noindent 
For a matrix $A\in \GL_d(\Z)$,   
the unitary $\U_A\in U(\ell^2(\Z^d))$  given by
$$ (\U_A \xi)(x)=\xi(A^{-1}x)$$
for $\xi \in \ell^2(\Z^d)$ and $x\in \Z^d$, 
defines an automorphism $\Ad \U_A$ of $B(\ell^2(\Z^d))$. 
Any restrictions of $\Ad \U_A$ to subalgebras of $B(\ell^2(\Z^d))$ 
will also be written as $\Ad \U_A$.

\vskip 1pc

\begin{remark}\label{isomorphisms} 
Consider  the $d$-torus 
$\A_\Theta=\overline{\rm span}\{l_\Theta (x):x\in \Z^d\}\subset B(\ell^2(\Z^d))$ 
associated with  
$\Theta \in \TT_d(\R)$. 
\begin{enumerate} 
\item[(1)] 
If $A\in \GL_d(\Z)$ satisfies $\Theta= {(A^{-1})}^t\Theta A^{-1}$, 
it defines an automorphism $\Ad U_A\in \Aut(\A_\Theta)$ 
because $\Ad U_A(l_\Theta (y))=l_{(A^{-1})^t\Theta A^{-1}} (Ay)\in \A_\Theta$ for 
the generators $\{l_\Theta (y): y\in \Z^d\}$ of $\A_\Theta$. 
In fact,  for  $\xi\in \ell^2(\Z^d)$ and $x,y\in \Z^d$,  
\begin{align*}
\Ad\U_A(l_\Theta(y))(\xi)(x)
           &=(\U_A  l_\Theta(y) \U_A^*)(\xi)(x)\\
           &=(l_\Theta(y) \U_A^*)(\xi)(A^{-1}\!x)\\
           &=\omega_\Theta(y,\!-y\!+\!A^{-1}\!x)(\U_A^*(\xi))(-y\!+\!A^{-1}\!x)\\
           &=\omega_\Theta(y,\!-y\!+\!A^{-1}\!x)\,\xi(\!-Ay\!+\!x)\\
           &=\omega_\Theta(y,A^{-1}\!x)\,\xi(\!-Ay\!+\!x)\\
           &=\omega_{{(A^{-1})}^t \Theta A^{-1}}(Ay,x)\,\xi(\!-Ay\!+\!x)\\
           &=(l_{{(A^{-1})}^t\Theta A^{-1}}(Ay))(\xi)(x). 
\end{align*} 
Thus we are concerned with the group $G_\Theta$.  

\item[(2)] If $\Theta=(\theta_{ij}) \in \TT_2(\R)$ with $\theta:=\theta_{12}\neq 0$, 
then $G_\Theta =\SL_2(\Z)$, which is immediate from the fact that 
$(A^{-1})^t\Theta A^{-1}= \det(A)\Theta$ for $A\in \GL_2(\Z)$.

\item[(3)] More generally, any matrix $A\in \GL_d(\Z)$ satisfying 
$$K_A:=\Theta-(A^{-1})^t \Theta A^{-1}\in \TT_d(\Z)$$ 
defines an automorphism $\tau_{K_A}\circ  \Ad U_A\in \Aut(\A_\Theta)$ such that  
$$(\tau_{K_A}\circ  \Ad U_A)(l_\Theta (y))=l_\Theta(Ay)$$ for $y\in \Z^d$, where 
$\tau_{K_A}(l_{(A^{-1})^t \Theta A^{-1}}(y)):=l_{(A^{-1})^t \Theta A^{-1}+K_A}(y)$.   
The set 
$\{\,A\in \GL_d(\Z)\,:\, \Theta-(A^{-1})^t \Theta A^{-1}\in \TT_d(\Z)\,\}$ 
forms a group. 
\end{enumerate}
\end{remark}

\vskip 1pc 

For each $\Theta\in \TT_d(\R)$, the group $G_\Theta$  
acts on  $\Z^d$ via matrix multiplication 
$$(A, x)\mapsto Ax:G_{\Theta} \times  \Z^d\to\, \Z^d$$ 
which then defines the semidirect product group $\Z^d\rtimes G_{\Theta}$ 
with the group multiplication   
$$(x,A)(y,B)=(x+Ay, AB)$$ 
for $x,y\in \Z^d$ and $A,B\in G_\Theta$.  
 Note that the cocycle $\omega_\Theta$, given in (\ref{omega}), is 
 {\it invariant} under the above action;
 $\omega_\Theta(Ax,Ay)=\omega_\Theta(x,y)$ 
 for $A\in G_\Theta$ and $x,y\in \Z^d$. 
  
The following lemma is a special case of 
\cite[Theorem 4.1]{PR} (see \cite[Lemma 2.1]{ELPW}).

\vskip 1pc 

\begin{lemma}\label{lemma2.1}
Let $\Theta\in \TT_d(\R)$ and $G$ be a subgroup of $G_\Theta$. Then:
\begin{enumerate}
\item[(1)] There is a $2$-cocycle $\widetilde\omega_\Theta$ 
of $\Z^d\rtimes G$ defined by 
\begin{equation}\label{omegatilde}
\widetilde\omega_\Theta((x,A), (y,B))=\omega_\Theta(x,Ay). 
\end{equation}

\item[(2)] There is an action $\af:G  \to \Aut(\A_\Theta)$ 
given by $\af_A(f)(x)=f(A^{-1}x)$ for 
$f\in \ell^1(\Z^d,\omega_\Theta)$ and $A\in G$,  
or equivalently 
$$\af_A(l_\Theta(x))=l_\Theta(Ax)  \  \ (l_\Theta(x)\in \A_\Theta).$$ 
 
\item[(3)] There are isomorphisms 
\begin{align*}
C^*(\Z^d\rtimes G, \widetilde\omega_\Theta) 
&\cong C^*(\Z^d,\omega_\Theta)\rtimes_\af G_,\\
C^*_r(\Z^d\rtimes G, \widetilde\omega_\Theta) 
&\cong C^*(\Z^d,\omega_\Theta)\rtimes_{\af,r} G
\end{align*}
given by 
 $f\mapsto \Phi(f): \ell^1(\Z^d\rtimes G, \widetilde\omega_\Theta)
 \to \ell^1 (G, \ell^1 (\Z^d,\omega_\Theta))$ on the level of $\ell^1$-functions, 
 where $\Phi(f)(A)=f(\cdot \,, A)$  for $A\in G$.
\end{enumerate}
\end{lemma}
 
\vskip 1pc

\begin{definition} \label{canonical automorphism}
The action $\alpha$ in Lemma~\ref{lemma2.1}(2),
 $$\af_A(l_\Theta (x))=l_\Theta(Ax), \ A\in G,\,  x\in \Z^d,$$ 
is called the {\it canonical action} of $G(\subset G_\Theta)$ on $\A_\Theta$. 
\end{definition} 

\vskip 1pc
 
Note that by (\ref{nc torus}) and Lemma~\ref{lemma2.1}(3), 
we have an isomorphism 
\begin{equation}\label{crossed product isomorphism}
\A_\Theta\rtimes_\af  G \cong C^*(\Z^d\rtimes G, \widetilde\omega_\Theta)
\end{equation} 
for any subgroup $G$ of $G_\Theta$ and its canonical action $\af$.  
 
\vskip 1pc

\begin{example} \label{action} 
For $\Theta=(\theta_{ij}) \in \TT_2(\R)$ with $\theta:=\theta_{12}$, 
the canonical action $\alpha$ of $G_\Theta$($= \SL_2(\Z)$ by 
Remark~\ref{isomorphisms}(2)) coincides with 
the action in (\ref{canonical action});  
if $A=(a_{ij})\in \SL_2(\Z)$, we have  for  $i=1,2$,
\begin{align*}
\alpha_A(l_\Theta(e_i))&=l_\Theta(Ae_i) =l_\Theta(a_{1i}e_1+ a_{2i}e_2)\\
    &=\exp(\pi i \theta a_{1i}a_{2i}) l_\Theta(e_1)^{a_{1i}} l_\Theta(e_2)^{a_{2i}}
    \, (\text{by }(\ref{expansion})).
\end{align*}
\end{example}

\vskip 1pc

\subsection{Companion matrices  of cyclotomic polynomials} 

To find a finite subgroup $G$ of $G_\Theta(\subset \GL_d(\Z))$, 
we have to examine matrices $A\in\GL_d(\Z)$ of finite order. 
For this, we shall use companion matrices of cyclotomic polynomials in Section 4 and 5. 

Consider a monic polynomial 
$p(x)=a_0+a_1x+\cdots+a_{d-1}x^{d-1}+x^d$. 
The {\it companion matrix} $C_{p(x)}$ of $p(x)$ is defined to be the following 
$d\times d$ matrix 
\begin{equation}\label{comp matrix}
C_{p(x)}:=\begin{pmatrix}
{0}&{0}&{0}&\cdots&{0}& {-a_0}\\
{1}&{0}&{0}&\cdots&{0}& {-a_1}\\
{0}&{1}&{0}&\cdots&{0}& {-a_2}\\[-7pt]
\vdots & \vdots& & \ddots& \vdots& \vdots\\[-7pt]
{0}&{0}&0 &\ddots&{0}&{-a_{d-2}}\\
{0}&{0}&{0}&\cdots&{1}& {-a_{d-1}}
\end{pmatrix} 
\end{equation}
which is invertible if $a_0\neq 0$.  
The minimal polynomial of $C_{p(x)}$ is equal to its 
characteristic polynomial $p(x)$.

Recall that for $n\in \N$, the $n$th {\it cyclotomic polynomial} 
$\Phi_n(x)$ is  defined by
\[
\Phi_n(x)=\prod_{\substack{1\leq k\leq n \\ {\rm gcd}(k,n)=1}} (x-\exp(2\pi i \,\frac{k}{n})).
\]
It is a monic polynomial of degree $d:=\phi(n)$ (here, $\phi$
is the Euler's totient function).  
$\Phi_n(x)$ is also known to have integer coefficients and is irreducible over $\Q$.   
The  companion matrix $C_n:=C_{\Phi_n(x)}$ of $\Phi_n(x)$ 
is then a  matrix of order $n$. 
 If  $n\geq 3$,  then $d$ is even and it is easy to see that  $C_n\in \SL_d(\Z)$,
  namely $\det(C_n)=1$.  
Since the minimal polynomial $\Phi_n(x)$ of $C_n$ 
has distinct  roots 
$\{\exp(2\pi i \,\frac{k}{n}):\, 1\leq k\leq n,\  {\rm gcd}(k,n)=1\}$  which 
are the eigenvalues of $C_n$, we see that   
$C_n$ is diagonalizable (in $\mathbb C$). 
Thus there is an invertible  matrix $U$ such that  
\begin{eqnarray} \label{diagonal matrix}
UC_n U^{-1}=\text{diag}(\zeta_1, \dots, \zeta_{\phi(n)}),
\end{eqnarray}
where 
$\zeta_1, \dots, \zeta_{\phi(n)}$ are the distinct primitive $n$-th roots of unity. 

\vskip 1pc 

\begin{remark}\label{semidirect group}
Let $n\in \mathbb N$ and $d:=\phi(n)$. 
Then the  companion matrix  $C_n$  of the cyclotomic polynomial $\Phi_n(x)$ 
generates the finite group $\mathbb Z_n=\langle C_n\rangle:=\{C_n^k\in \GL_d(\Z):\, 0\leq k\leq n-1\}$ 
which acts on the group $\mathbb Z^d$ via 
$$(C_n^k, x)\mapsto C_n^k\, x:\mathbb Z_n\times \mathbb Z^d\to \mathbb Z^d.$$  
This action (also denoted $\af$) is actually the conjugation action of $\mathbb Z_n$  
on $\mathbb Z^d$ in the semidirect product group  
$\Z^d\rtimes_\af  \mathbb Z_n$. 
\end{remark}
 
\vskip 1pc 

\subsection{Classification theorems}
 
\noindent 
Let $\Theta\in \TT_d(\R)$ be nondegenerate. 
Then $\A_\Theta$ is a simple $C^*$-algebra with 
a unique tracial state by \cite[Theorem 1.9]{Ph06} 
and has tracial rank zero by \cite[Theorem 3.5]{Ph06}. 
Thus 
if $\alpha:G\to Aut(\A_\Theta)$ is an action by a finite group which has 
the tracial Rokhlin property (see \cite[Section 5]{ELPW}), the crossed product 
$\A_\Theta \rtimes_\alpha G$ becomes a simple $C^*$-algebra (\cite[Corollary 1.6]{Ph11}) 
with tracial rank zero (\cite[Theorem 2.6]{Ph11}).  
The fact that $\A_\Theta \rtimes_\alpha G$ has a unique tracial state 
follows from \cite[Proposition 5.7]{ELPW}.
 
 The canonical action $\af$ of a finite group $G(\subset G_\Theta)$ 
 on the simple  $\A_\Theta$ is actually known to have the tracial Rokhlin property 
 by \cite[Lemma 5.10 and Theorem 5,5]{ELPW}, and  moreover 
the crossed product $\A_\Theta\rtimes_\af G$ satisfies 
the Universal Coefficient Theorem (this will be shown in Proposition~\ref{uct}). 
 Thus the crossed product $\A_\Theta\rtimes_\af G$ 
becomes classifiable by Huaxin Lin's classification theorem: 

\vskip 1pc

\begin{thm} {\rm (\cite[Theorem 5.2]{Lin05})}\label{Lin05} 
Let $A$ and $B$ be two unital separable simple nuclear $C^*$-algebras with 
tracial topological rank zero which satisfy the Universal Coefficient Theorem. 
Then $A\cong B$ if and only if they have isomorphic Elliott invariants, that is,
$$(K_0(A), K_0(A)_+, [1_A], K_1(A))\cong (K_0(B), K_0(B)_+, [1_B], K_1(B)).$$
\end{thm}

\vskip 1pc
\noindent
 Since simple unital AF algebras satisfy all the conditions of the above theorem, 
 if the Elliott invariant of $\A_\Theta\rtimes_\af G$, $G\subset G_\Theta$, 
 is isomorphic to that of such an  AF algebra,  
 one can conclude that the crossed product 
 $\A_\Theta\rtimes_\af G$ is an AF algebra, 
 which was successfully done in \cite{ELPW} for $\A_\theta\rtimes_\alpha F$ 
 with all $\theta\in \R\setminus \Q$ and all finite subgroups 
 $F$  of $\SL_2(\Z)$. 
The following proposition by N.C. Phillips also says that to see  
whether the crossed product $\A_\Theta\rtimes_\af G$ is AF, 
we only need to know its $K$-groups: 

\vskip 1pc

\begin{prop}{\rm ({\rm \cite[Proposition 3.7]{Ph06}})}\label{Phi K-groups} 
Let $\A$ be a simple infinite dimensional separable unital nuclear $C^*$-algebra with 
tracial rank zero and which satisfies the Universal Coefficient Theorem. 
Then $\A$ is a simple AH algebra with real rank zero and no dimension growth. 
If $K_*(\A)$ is torsion free, $\A$ is an AT algebra. If, in addition, 
$K_1(\A)=0$, then $\A$ is an AF algebra.
\end{prop}
 
\vskip 1pc 

\begin{remark}\label{conclusion} 
We can summarize what the classification results above 
together with Proposition~\ref{uct} imply in our setting as follows: 
If $\Theta$ is a nondegenerate skew symmetric real 
$d\times d$ matrix  and $\af:G\to  Aut (\A_\Theta)$ is the 
canonical action of a finite group $G\subset G_\Theta$,  
the simple crossed product $\A_\Theta\rtimes_\af G$ is an AF algebra 
if and only if 
$K_0( \A_\Theta\rtimes_\af G)$ is torsion free and 
$K_1( \A_\Theta\rtimes_\af G)=0$.
\end{remark}

\vskip 1pc

For the computation of $K_1$-groups of the crossed products 
$\A_\Theta\rtimes_\af G$, 
we shall apply the following theorem.

\vskip 1pc 

\begin{thm}{\rm (\cite[Theorem 0.1]{LL12}, \cite[Theorem 0.3]{DL})}\label{LL12} 
Let  $n,\, d\in \mathbb N$. Consider the extension of groups 
$1\to \Z^d\to \Z^d\rtimes_\af \Z_n\to \Z_n\to 1$  such that 
conjugation action $\af$ of $\Z_n$ on $\Z^d$ is free outside the origin 
$0\in \Z^n$.  Then 
$ K_0(C^*(\Z^d\rtimes_\alpha \Z_n)) \cong \Z^{s_0}$ for some $s_0\in \Z$ and 
$$ K_1(C^*(\Z^d\rtimes_\alpha \Z_n)) \cong \Z^{s_1},$$
where 
$ s_1=\sum_{l\geq 0} \rk_\Z((\Ld^{2l+1}\Z^d)^{\Z_n})$. 
If $n$ is even,  $s_1=0$. 
If $n>2$ is prime and $d =n-1$, then $s_1=\frac{2^{n-1}-(n-1)^2}{2n}$. 
\end{thm}

\vskip 1pc

\section{$K$-groups of the simple crossed products $\A_\Theta\rtimes_\af \Z_n$}

\noindent 
In this section we show that the  
$K_1$-group  of the simple crossed product $\A_\Theta\rtimes_\af G$  
is not always zero (see Theorem~\ref{nonzero K_1}).
We begin with the following proposition saying that we can apply 
Lin's classification theorem (Theorem~\ref{Lin05})
or Proposition~\ref{Phi K-groups}
to the simple crossed product  $\A_\Theta \rtimes_\af G$ 
of the canonical action by a finite subgroup $G$ of $G_\Theta$.

\vskip 1pc  
 
\begin{prop}\label{uct}
Let $\Theta$ be a skew symmetric real $d\times d$ matrix and 
$G$ be a finite subgroup of $G_\Theta$. 
If $\alpha:G \rightarrow \Aut(\A_\Theta)$ is the canonical action of $G$ 
on $\A_\Theta$, the crossed product 
$\A_\Theta \rtimes_\alpha G$ satisfies the Universal Coefficient Theorem. 
\end{prop}
\begin{proof}
The 2-cocycle 
$\omega_\Theta$ given in (\ref{omega}) is invariant under the action of $G$ on $\Z^d$. 
By (\ref{crossed product isomorphism}), the crossed product  
$\A_\Theta \rtimes_\alpha G$ is isomorphic to the twisted group algebra 
$C^*(\Z^d\rtimes G,\widetilde{\omega}_\Theta)$. 
Note that $\Z^d\rtimes G$ is amenable and is  
a closed subgroup of $\R^d\rtimes G$ which is almost connected. 
Therefore $\A_\Theta \rtimes_\alpha G$ satisfies the Universal Coefficient Theorem  
(see \cite[Corollary 6.2]{ELPW}). 
\end{proof}
 
\vskip 1pc

\begin{remark} \label{K-groups}
If there is a finite subgroup $G$ of $G_\Theta$ which canonically acts
on the noncommutative simple torus $\A_\Theta$, 
as summarized in Remark~\ref{conclusion}, we need to calculate the $K$-groups 
of $\A_\Theta \rtimes_\af G$ to see whether it is an  AF algebra. 
But the $K$-groups 
$K_*(\A_\Theta \rtimes_\alpha G)$ of the crossed product  
$\A_\Theta \rtimes_\alpha G\cong C^*(\Z^d\rtimes G,\widetilde{\omega}_\Theta)$ 
(see (\ref{crossed product isomorphism})) 
is equal to  the $K$-groups $K_*(C^*(\Z^d\rtimes G))$ of the 
untwisted group algebra by \cite[Theorem 0.3]{ELPW};
$$ K_i(\A_\Theta \rtimes_\alpha G)=K_i(C^*(\Z^d\rtimes G))$$
for $i=0,1$. 
This is because the 2-cocycle $\widetilde{\omega}_\Theta$ is homotopic 
(in the sense of \cite[Theorem 0.3]{ELPW}) to 
the trivial one
via
$$ \Omega:(\Z^d\rtimes G)\times (\Z^d\rtimes G) \rightarrow C([0,1],\T) $$
defined by 
$$ \Omega((x,A),(y,B))(t):=\exp(2\pi i t \langle \Theta x, Ay \rangle ) $$
for $x,y\in \Z^d$, $A,B\in G$ and $t\in [0,1]$.   
\end{remark}

\vskip 1pc 

For the rest of this section, we will consider the cyclic group $\Z_n=\langle C_n\rangle$  
generated by the companion matrix $C_n$ 
of the $n$th cyclotomic polynomial and its conjugation action  $\af$ on $\Z^d$,  
$d=\phi(n)$ (see  Remark~\ref{semidirect group} for the conjugation action). 
Note that we use the same $\af$   for both  the canonical action of $G(\subset G_\Theta)$ on 
$\A_\Theta$ and the conjugation action of $\Z_n$ on $\Z^d$. 

\vskip 1pc 

The formula $K_*(C^*(\Z^d \rtimes_\af  G))$ in Theorem~\ref{LL12} requires 
that the action $\af$ of $\Z_n$ on $\Z^d$ be free outside the origin, 
so we need the following proposition which follows immediately from (\ref{diagonal matrix}).

\vskip 1pc 

\begin{prop}\label{away 0} The conjugation action $\af$ of   
$\mathbb Z_n=\langle C_n\rangle$ on 
$\mathbb Z^d$ in the semidirect product group $\Z^d\rtimes_\af \mathbb Z_n$ 
is {\it free outside the origin}, that is, 
$C_n^k\, x\neq x$ for all $k=1,\dots, n-1$ and nonzero $x\in \mathbb Z^d$. 
\end{prop}

\vskip 1pc

\noindent 
The above proposition, together with 
Proposition~\ref{Phi K-groups}, Remark~\ref{K-groups}, and 
Theorem~\ref{LL12}, gives the following:

\vskip 1pc 

\begin{prop}\label{AT} 
Let  $d=\phi(n)$ and $\af$ be the canonical action of the finite cyclic group 
$\Z_n=\langle C_n \rangle$ on a noncommutative simple $d$-torus $\A_\Theta$.
Then the crossed product $\A_\Theta\rtimes_\af \Z_n$ is an AT algebra, 
and moreover it is an AF algebra 
if and only if $K_1(\Z^d\rtimes_\af \Z_n)=0$;  
in particular if $n$ is even, $\A_\Theta\rtimes_\af \Z_n$ is always an AF algebra.
\end{prop}

\vskip 1pc

Now  let $n\geq 3$ be an odd number and $\af$ be the conjugation action 
of $\Z_n=\langle C_n\rangle$ on $\Z^d$. 
We will show that $K_1(\Z^d\rtimes_\af \Z_n)$ is not necessarily zero. 
For each $l\geq 0$,  $\af$  induces an action 
$\Ld^l(\af): \Z_n  \to \Aut(\Ld^l\Z^d)$  of $\Z_n$ 
on the $l$th exterior power $\Ld^l\Z^d$ of $\Z$-module $\Z^d$ as follows:   
\begin{align*}
\Ld^l(\af)(k)(x_1\w \cdots \w x_l)& :=\Ld^l(\af_k)(x_1\w \cdots \w x_l)\\
&=\af_k(x_1)\w\cdots \w \af_k(x_l) 
\end{align*}
for $k\in \Z_n$ and $x_1\w \cdots \w x_l\in \Ld^l\Z^d$. 
For notational convenience, we simply write $\Ld(\af)$ for $\Ld^l(\af)$. 
 To compute $K_1(C^*(\Z^d\rtimes_\af \Z_n))$,
 we need to know the rank  $\rk_\Z(\Ld^l\Z^d)^{\Z_n}$ of the following submodule  
$$(\Ld^l\Z^d)^{\Z_n}:=\{ v \in \Ld^l\Z^d : \Ld(\alpha_k)(v)=v,\ k\in \Z_n \}$$ 
of the fixed points.

\vskip 1pc 

\begin{lemma}\label{l=d}
Let $n\geq 3$ be an odd integer and $d:=\phi(n)$ 
(automatically even).
Then $\Ld^d\Z^d=(\Ld^d\Z^d)^{\Z_n}$. 
\end{lemma}
\begin{proof}
It is enough to show that $\Ld(\alpha_k)(e_1\w\cdots\w e_d)=e_1\w\cdots\w e_d$ 
for all $k\in \Z_n$, which is obvious from  
$\Ld(\alpha_k)(e_1\w\cdots\w e_d) =\alpha_k(e_1)\w\cdots\w\alpha_k(e_d) 
=C_n^ke_1\w\cdots\w C_n^ke_d =\det(C_n^k)e_1\w\cdots\w e_d$ and 
the fact that $\det(C_n)=1$.
\end{proof}

\vskip 1pc 
 
For  two subsets $I, J$ of $\Z$,  set 
$$J-I :=\{j-i \in \Z : i\in I, j\in J \}$$ 
and write $I\equiv J\, (\hskip -.5pc \mod n)$ if for each $i\in I$ there exists a 
$j\in J$ with $i-j \in n\Z$ and vice versa. 
As usual, $|I|$ denotes the cardinality of the set $I$.

\vskip 1pc 

\begin{thm}\label{nonzero K_1} 
Let $n\geq 7$ be an odd integer and  $d:=\phi(n)$. 
Consider the extension of groups 
$1\to \Z^d\to \Z^d\rtimes_\af \Z_n\to \Z_n\to 1$ with $\Z_n=\langle C_n\rangle$. 
 If $2d\geq n+5$, then  
$$K_1(C^*(\Z^d\rtimes_\af \Z_n))\neq 0.$$ 
$($If $n\geq 7$ is prime,  $2d \geq n+5$ always holds.$)$
\end{thm}

\begin{proof}
Since the conjugation action $\af$ of $\Z_n=\langle C_n\rangle$ on $\Z^d$ is free outside the 
origin by Proposition~\ref{away 0}, 
 it is enough to show that 
$s_1=\sum_{l\geq 0} \rk_\Z(( \Ld^{2l+1}\Z^d)^{\Z_n} )\neq 0$ by Theorem~\ref{LL12}.

Note first that the set $\{1, 2, \dots, d\}$ can be divided into 
two disjoint sets $I=\{i_1, \dots, i_l\}$ and $J=\{j_1, \dots, j_{d-l}\}$ 
such that   $|I|$ (hence $|J|$) is odd and       
\begin{align*}
J- I& \equiv \{1,2,\dots,d-2\}\cup \{n-d+3, n-d+4,\dots,n-1\}\ (\hskip -.8pc\mod n) \\
    & \equiv \{1,2,\dots,n-1\}\ (\hskip -.8pc\mod n), 
\end{align*}
which follows from the condition $2d\geq n+5$ (or $d-2\geq n-d+3$). 
(For example, one can take  $I=\{1,2,d\}$ and $J=\{3,4,\dots,d-1\}$.) 
Then for any  $k,t\in \Z_n=\{0,1,\dots,n-1\}$ with $k\neq t$, 
there exist $i\in I$ and $j\in J$ such that $k+i \equiv t+j\ (\hskip -.6pc\mod n)$. 
So we have 
\begin{align*}
&\sum_{0\leq k\neq t \leq n-1}\Ld(\alpha_k)(e_{i_1}\w\dots\w 
                   e_{i_l})\w \Ld(\alpha_t)(e_{j_1}\w\dots\w e_{j_{d-l}})\\
=&\sum_{0\leq k\neq t \leq n-1} C^k_ne_{i_1}\w\dots\w 
                   C^k_n e_{i_l}\w C^t_n e_{j_1}\w\dots\w C^t_n e_{j_{d-l}} \\
=&\sum_{0\leq k\neq t \leq n-1}C^{k+{i_1}-1}_ne_1\w\dots \w 
                   C^{k+{i_l}-1}_ne_1\w C^{t+{j_1}-1}_ne_1\w\dots\w C^{t+j_{d-l}-1}_ne_1\\
=&\ 0,
\end{align*}
where the second equality  comes from the easy fact that $C_n^{v-1}e_1=e_v$ for $v=1,\dots,d$. 
By Lemma \ref{l=d},  
$\Ld(\af_k)(e_{i_1}\w\cdots\w e_{i_l}\w e_{j_1}\w \cdots\w e_{j_{d-l}})
=e_{i_1}\w\cdots\w e_{i_l}\w e_{j_1}\w \cdots\w e_{j_{d-l}}$ 
for all $k$. 
Thus 
\begin{align*}
0\neq&\ \ n (e_{i_1}\w\dots\w e_{i_l}\w e_{j_1}\w \dots\w e_{j_{d-l}})\\
=&\,\ \sum_{k=0}^{n-1}\Ld(\alpha_k)(e_{i_1}\w\dots\w 
                      e_{i_l}\w e_{j_1}\w \dots\w e_{j_{d-l}})\\
=&\,\ \big(\,\sum_{k=0}^{n-1}\Ld(\alpha_k)(e_{i_1}\w\dots\w e_{i_l}))\w
                      \big(\,\sum_{t=0}^{n-1}\Ld(\alpha_t)( e_{j_1}\w \dots\w e_{j_{d-l}})\big)\\
&\ -\sum_{0\leq k\neq t \leq n-1}\Ld(\alpha_k)(e_{i_1}\w\dots\w 
                      e_{i_l})\w \Ld(\alpha_t)(e_{j_1}\w\dots\w e_{j_{d-l}})\\
=&\,\ \big(\,\sum_{k=0}^{n-1}\Ld(\alpha_k)(e_{i_1}\w\dots\w e_{i_l})\big)\w 
                      \big(\,\sum_{t=0}^{n-1}\Ld(\alpha_t)( e_{j_1}\w \dots\w e_{j_{d-l}})\big),
\end{align*}
and hence
$\displaystyle \sum_{k=0}^{n-1}\Ld(\alpha_k)(e_{i_1}\w\dots\w e_{i_l})\neq 0$ and 
$\displaystyle \sum_{t=0}^{n-1}\Ld(\alpha_t)( e_{j_1}\w \dots\w e_{j_{d-l}})\neq 0$. 
But clearly these two elements belong to $(\Ld^l \Z^d)^{\Z_n}$ and 
$(\Ld^{d-l}\Z^d)^{\Z_n}$ respectively, so that  
$s_1>0$ follows.
\end{proof}

\vskip 1pc 

We close this section with providing a condition equivalent to 
$2d\geq n+5$ used in Theorem~\ref{nonzero K_1}.

\vskip 1pc 

\begin{prop}  Let $n\geq 7$ be an odd number and $d:=\phi(n)$. Then
the condition  $2d\geq n+5$ of Theorem~\ref{nonzero K_1} holds 
if and only if there is a partition 
$\{I, J\}$ of $\{1,\dots,d\}$ such that both $|I|$ and $|J|$ are odd  and 
$$ 
J-I \equiv \{1,2,\dots,n-1\} \ (\hskip -.8pc\mod n).$$
\end{prop}

\begin{proof} 
The direction `only if' was proved in the proof of Theorem~\ref{nonzero K_1}. 
For the converse, let  $\{I, J\}$ be such a partition of $\{1,\dots, d\}$ 
that $|I|$ and $|J|$ are odd and $J-I\equiv \{1,2,\dots, n-1\} \ (\hskip -.5pc\mod n).$  
Assume $1\in I$.
Then the sets
$$P:=(J-I)\cap \N\ \text{ and }\ N:=(J-I)\setminus P.$$ 
contain the positive and negative integers of $J-I$ respectively.    
With $N':=\{n+m:m\in N\}$ of positive integers, 
one has $N'\equiv \, N (\hskip -.7pc\mod n)$. 
Also 
\begin{equation}\label{LR}
P\subset \{1,2,\dots,d-1\}\ \text{ and } N'\subset  \{n+(1-d),n+(2-d),\dots,n -1\}
\end{equation} 
is clear, hence $P\cup N' \subset \{1,2,\dots,n-1\}$ and 
$|P|, |N'|\leq d-1$. 
From 
$$J-I=P\cup N \equiv P\cup N' \, (\hskip -.8pc \mod n) \ \text{ and } 
J-I\equiv \{1,2,\dots,n-1\} \, (\hskip -.8pc\mod n),$$ it follows that 
\begin{equation}\label{PcupN'}
P\cup N'= \{1,2,\dots, n-1\}
\end{equation}
since $P\cup N'$ has only positive integers less than $n$. 
Thus $n-1\leq |P|+|N'|\leq 2(d-1)$, namely $2d\geq n+1$ must hold. 
Since $n$ is odd, we have $2d=n+1$, $2d=n+3$ or $2d\geq n+5$.   

To show $2d\geq n+5$, first suppose  $2d=n+1$. Then $2(d-1)=n-1$ so that
we have $N'=\{n+1-d,n+2-d,\dots,n-1\}$ 
and thus $1-d\in N$. But $1-d$ is not equal to any number 
$j-i$ for $j\in J$ and $i\in I$. Thus $2d\neq n+1$.

Now suppose that $2d=n+3$, that is, $2(d-1)=n+1$ holds. 
Set 
$$L:=\{1,2,\dots,d-1\}\ \text{ and }\ R:=\{n+1-d,n+2-d,\dots,n-1\},$$ 
then  $2d=n+3$ is the case exactly when 
\begin{equation}\label{LcapR}
L\cap R=\{d-2, d-1 \}=\{n+1-d, n+2-d\}.
\end{equation} 
Also (\ref{PcupN'}) implies that 
\begin{equation}\label{0411_1}
 \{1,2,\dots,d-3\}\subset P\ \text{ and } \{n+3-d, n+4-d,\dots,n-1\}\subset N'.
\end{equation} 
Note here that if $m\in L\setminus R=\{1,\dots, d-3\}$ then $m\in J-I$ and that 
if $m\in R\setminus L$ then $m-n\in J-I$. 
Moreover each $m\in L\cap R$ satisfies the following:
\begin{equation}\label{0411_3}
m\notin J-I \Rightarrow m-n\in J-I.
\end{equation}
We claim that 
$$\{1, d \}\subset  I\ \text{ and }\   \{2, d-1\}\subset J.$$  
First observe that for $m:=n+1-d \in L\cap R$,   $m-n=1-d\notin J-I$ since $1\in I$. 
By (\ref{0411_3}), $m=n+1-d\in J-I$ and hence $d-2\in J-I$ by (\ref{LcapR}).
Thus we have
\begin{equation}\label{0411_2}
(d,2)\in J\times I \ \text{ or }\ (d-1,1)\in J\times I.
\end{equation}
Since $3-d\in J-I$ by (\ref{0411_1}), we should have at least one of the following;
$$(3,d)\in J\times I,\  (2,d-1)\in J\times I,\ \text{ or }\  (1,d-2)\in J\times I.$$ 
But  $(2,d-1)\notin J\times I$ and $(1,d-2)\notin J\times I$  
because of (\ref{0411_2})  and  our assumption that $1\in I$. 
Thus  $(3,d)\in J\times I$. 
Since $d\notin J$ we also have $d-1\in J$ by (\ref{0411_2}). 
From (\ref{LcapR}) $d-1=n+2-d$ and thus  
we have $n+2-d=d-1 \notin J-I$ (otherwise, $d\notin J$). 
Then by (\ref{0411_3}), $2-d\in J-I$ which can occur only when 
$(2,d)\in J\times I$ or $(1,d-1)\in J\times I$. 
Since $1\in I$, we obtain $2\in J$. 
This proves the claim. 
 Next we show that 
$$\{1, d \}\subset  I\ \text{ and }\   \{2, 3, d-2, d-1\}\subset J$$  
providing a proof that can be repeated until we reach the 
step $\{1,d\}=I$ and  $\{2, \dots,  d-1\}= J$,  
where we meet a contradiction to the assumption that both 
$I$ and $J$ have odd number of elements. 
By (\ref{0411_1}), with $k=2$,  $n+(k+1)-d\in N'$ and $d-(k+1)\in P$.  
Thus $$(k+1)-d\in J-I \text{ and }\ d-(k+1)\in J-I.$$ 
Note that $(k+1)-d=3-d\in J-I$ can happen only if  
$(k+1,d)=(3,d)\in J\times I$, $(k,d-1)=(2,d-1)\in J\times I$, or 
$(1,d-k)=(1, d-2)\in J\times I$. 
But obviously  $(k+1,d)=(3,d)\in J\times I$ is the only possible case 
and we have $k+1=3\in J$.
Since $d-(k+1)=d-3\in J-I$ we obtain $\{2,3,d-2,d-1\}\in J$. 
(One can repeat the same argument on $k$.)

So far we have shown that the cases 
$2d=n+1$ and $2d=n+3$ should be excluded from $2d\geq n+1$, 
which gives $2d\geq n+5$ as desired. 
\end{proof}

\vskip 1pc

\section{Finite cyclic groups acting on  
higher dimensional noncommutative simple tori $\A_\Theta$}

\noindent 
 We have seen in the previous section 
 that if $\af$ is a canonical action of $\Z_n=\langle C_n\rangle$ on a 
 higher dimensional simple $d$-torus $\A_\Theta$, then we are well informed about
 the crossed products $\A_\Theta\rtimes_\af\Z_n$ 
 since we know how to compute the decisive invariants, namely   
 the $K$-groups $K_*(\A_\Theta\rtimes_\af\Z_n)$ which are  
 equal to $K_*(C^*(\Z^d\rtimes_\af \Z_n))$ (Remark~\ref{K-groups}) 
 and are not necessarily zero by Theorem~\ref{nonzero K_1}.
So it seems that many of the crossed products $\A_\Theta \rtimes_\af \Z_n$ are 
far from being AF.
But what is not yet clear is if there does exist any   
noncommutative simple $d$-torus $\A_\Theta$ that 
actually admits the canonical action by a finite group. 
In this section we show that there do really exist such 
higher dimensional simple tori and then 
determine when their crossed products are not AF.
 To explain more precisely what we do here, 
 we remark that we do not try to find  finite groups $G$  
acting on a  torus $\A_\Theta$ with $\Theta$ preassigned because 
this way of finding examples does not seem to work effectively, rather  
if we begin with a finite (cyclic) group $G$ 
generated by a matrix $C\in \GL_d(\Z)$ of finite order, 
it is more easily addressable to find $d$-tori $\A_\Theta$ on which $G$ acts canonically. 
This is our strategy to investigate examples  in this paper, and so 
we use the following notation for convenience' sake.
 
\vskip 1pc 

\begin{ntt}\label{TT-A}
For $A\in \GL_d(\Z)$, we set: 
$$
\TT_{d,A}(\R) :=\{\Theta \in \TT_d(\R): \Theta = A^t \Theta A \,\}.
$$
$\TT_{d,A}(\R)$ is the set of all skew symmetric matrices $\Theta$ such that 
the noncommutative tori $\A_\Theta$ admit the canonical action of 
the group $\langle A\rangle$ generated by $A$. 
For $A\in \GL_d(\Z)$ and $\Theta\in \TT_d(\R)$, 
it is obvious that $A\in G_\Theta$ if and only if 
$\Theta\in \TT_{d, A}(\R)$.
\end{ntt}

\vskip 1pc 

 Since the companion matrix $C_n$ is of order $n$, we see that  
 any skew symmetric matrix 
 $\displaystyle \Theta:=\sum_{k=0}^{n-1} (C_n^k)^t \Theta' C_n^k $ 
 (for $\Theta'\in \TT_d(\R)$)
 satisfies the condition $C_n^t \Theta C_n=\Theta$ to be an element of 
 $\TT_{d, C_n}(\R)$. 
Conversely,  if $\Theta\in \TT_{d,C_n}(\R)$, that is $C_n^t \Theta C_n=\Theta$,  then 
$\Theta= \frac{1}{n}\sum_{k=0}^{n-1} (C_n^k)^t \Theta C_n^k$ is rather clear. 
Thus we see that 
$$\TT_{d,C_n}(\R)=\Big\{\, \sum_{k=0}^{n-1} (C_n^k)^t \Theta C_n^k:\, \Theta\in  \TT_d(\R)\,\Big\}.$$

Now we show that for each $n\geq 3$, the cyclic group $\Z_n$ acts canonically on 
some noncommutative simple $\phi(n)$-dimensional tori $\A_\Theta$. 
Note from the  following theorem 
that since $\phi(n)=p_1^{r_1-1}(p_1-1)\cdots p_s^{r_s-1}(p_s-1)$ when  
$n=p_1^{r_1}\cdots p_s^{r_s}$ is the prime factorization of $n$, 
the group $\Z_n$  always acts on some noncommutative 
higher dimensional simple tori whenever $n=5$ or $n\geq 7$.

\vskip 1pc

\begin{thm}\label{nondegenerate theta} Let $n\geq 3$ and $d:=\phi(n)$. 
Then there exist simple $d$-dimensional tori $\A_\Theta$ on which 
the group $\Z_n=\langle C_n\rangle$ acts canonically.
\end{thm} 
\begin{proof} 
We show that $\TT_{d,C_n}(\R)$ contains nondegenerate matrices. 
Let $\theta$ be an irrational number. 
Then
$\displaystyle \Theta:=\theta \sum_{k=0}^{n-1} (C_n^k)^t  (C_n^t-C_n) C_n^k\ $ 
is a skew symmetric matrix in $\TT_{d,C_n}(\R)$. 
To show that $\Theta$ is nondegenerate, suppose 
$ \Theta x \in \Z^d$ for some $x\in \Z^d$. 
Since the entries of $C_n$ are integers and $\theta$ is irrational, 
we must have 
$\big(\sum_{k} (C_n^k)^t  (C_n^t-C_n) C_n^k\,\big)x =0$ in $\Z^d$. 
Then  
\begin{align*}
0 &=\ \Big(\sum_{k=0}^{n-1} (C_n^k)^t  (C_n^t-C_n) C_n^k  \Big)\, x\\
  &= \Big(C_n^t \sum_{k=0}^{n-1} (C_n^k)^t C_n^k- (\,\sum_{k=0}^{n-1} (C_n^k)^t C_n^k\,) C_n \Big)\, x \\
  &=\, \Big(C_n^t \sum_{k=0}^{n-1} (C_n^k)^t C_n^k- C_n^t(\,\sum_{k=0}^{n-1} (C_n^k)^t C_n^k\,) C_n^2 \Big)\, x\\
  &=\ C_n^t\sum_{k=0}^{n-1}(C_n^k)^t C_n^k \,(I_d-C_n^2)\, x,
\end{align*}
 where $I_d$ is the $d\times d$ identity matrix.
 Thus $(I_d-C_n^2)x$ must be zero 
because the matrix $C_n^t\sum_{k}(C_n^k)^t C_n^k $ is  invertible. 
But, as we have seen in Proposition~\ref{away 0}, 
$x\neq C_n^2 x$  for nonzero $x\in \Z^d$, and we conclude that
$\Theta$ is nondegenerate.
\end{proof}

\vskip 1pc 

In case that $n(\geq 3)$ is prime, we especially can find all the 
skew symmetric matrices $\Theta\in \TT_{n-1}(\R)$ such that 
the noncommutaive tori $\A_\Theta$ 
associated with $\Theta$ admit the canonical action of the group
$\Z_n=\langle C_n\rangle$.
 We begin with finding the general form of $\Theta$ for which 
the cyclic group generated by a companion matrix $C$ of 
finite order can possibly act on $\A_\Theta$.

\vskip 1pc 

\begin{lemma}\label{theta lemma}  
Let $\Theta\in \TT_d(\R)$ be a nonzero skew symmetric matrix and let 
$C \in \GL_d(\Z)$ be the companion matrix of 
a monic polynomial with degree $d(\geq 3)$. 
Assume that $C$ is of order $n$ and  
the noncommutative  $d$-torus $\A_\Theta$  
admits the canonical action by $\Z_n=\langle C \rangle$, then 
$\Theta$ has the following form:  
\begin{equation}\label{general theta form}
\Theta =\begin{pmatrix}
0 &\theta_0 &\theta_1 &\theta_2 &\cdots &\theta_{d-3} & \theta_{d-2}\\[4pt]
& 0 &\theta_0 &\theta_1 &\theta_2 &\cdots &\theta_{d-3}\\
& & 0 &\theta_0 &\theta_1 &\ddots &\vdots\\ 
& & &0 &\theta_0 &\ddots &\theta_2\\
& & & &0 &\ddots &\theta_1\\
& & & & &\ddots &\theta_0 \\[4pt]
& & & & & &0
\end{pmatrix}  
\end{equation}
for $\theta_i\in \R$, $i=0,\dots, d-2$.
\end{lemma}

\begin{proof} 
Recall that  $\Z_n=\langle C \rangle$ canonically acts on 
$\A_\Theta$ if and only if $C \in G_\Theta$, where 
$G_\Theta=\{A\in \GL_d(\Z): A^t\Theta A =\Theta\}$.
Let $\Theta_i$ and $\Theta^j$ be the $i$-th row and $j$-th column of 
$\Theta=(\theta_{ij})$, respectively for $i,j=1, \dots, d$. 
Let $C$ be the companion matrix of 
a monic polynomial $a_1+a_2x+\cdots a_{d}x^{d-1} +x^d$ of degree $d$ 
($a_1\neq 0$ because $C$ is invertible). 
Then we can write $C$ (see (\ref{comp matrix}) as the sum 
$C=A+B$ of two matrices $A$ and $B$, where 
$$
A=\begin{pmatrix} 
0 & 0 & 0 & \cdots & 0 \\
1 & 0 & 0 & \cdots & 0 \\
0 & 1 & 0 & \cdots & 0  \\ 
\vdots &  & \ddots & \ddots & \vdots \\
0 & 0 & \cdots & 1 & 0  
\end{pmatrix}\   \text{ and }\ 
B=\begin{pmatrix} 
0 & 0 & \cdots & 0 & a_1\\
0 & 0 & \cdots & 0 & a_2\\
0 & 0 & \cdots & 0 & a_3\\ 
\vdots & \vdots & \ddots & \vdots & \vdots\\
0 & 0 & \cdots & 0 & a_d\\
\end{pmatrix}.
$$ 
Then a computation, with ${\bf a}=(a_1, \dots, a_d)^t$, shows that 
\begin{align*}
C^t\Theta C 
& \ = \ A^t\Theta A + A^t\Theta B  + B^t \Theta A +B^t\Theta B\\
& \ = \  \begin{pmatrix} 
 \theta_{22} & \cdots & \theta_{2d} & 0 \\
 \vdots & \ddots  &\vdots & 0\\
 \theta_{d2} & \cdots & \theta_{d,d} & 0 \\
 0 &  \cdots & 0 & 0  
\end{pmatrix}
+
\begin{pmatrix} 
0 & \cdots & 0 &\Theta_2  \,{\bf a} \\
\vdots & &\vdots &\vdots \\
0 &  \cdots & 0 & \Theta_d \, {\bf a} \\
0 &  \cdots & 0 & 0  
\end{pmatrix}\\
& \ + 
\begin{pmatrix} 
0 & \cdots & 0 & 0 \\
\vdots & &\vdots &\vdots \\
0 &  \cdots & 0 & 0 \\
(\Theta^2)^t\, {\bf a}  &  \cdots & (\Theta^d)^t\, {\bf a} & 0  
\end{pmatrix}
+
\begin{pmatrix} 
0 & \cdots & 0 & 0 \\
\vdots & &\vdots &\vdots \\
0 &  \cdots & 0 & 0 \\
0 &  \cdots & 0 &   {\bf a}^t\Theta\, {\bf a}
\end{pmatrix}.
\end{align*}
The assertion then follows 
from the fact that $C^t\Theta  C$ is equal to $\Theta$.
\end{proof}

\vskip 1pc 

For a prime $p\geq 3$,  consider the 
skew symmetric  $(p-1)\times (p-1)$ matrices $\Theta$ of the following form: 
\begin{equation}\label{prime theta form}
\Theta =\begin{pmatrix}
0 &\theta_0 &\theta_1 &\theta_2 &\cdots &-\theta_{2} & -\theta_{1}\\[3pt]
& 0 &\theta_0 &\theta_1 &\theta_2 &\cdots &-\theta_{2}\\
& & 0 &\theta_0 &\theta_1 &\ddots &\vdots\\ 
& & & 0 &\theta_0 &\ddots &\theta_2\\
& & & &0 &\ddots &\theta_1\\
& & & & &\ddots &\theta_0 \\[2pt]
& & & & & &0
\end{pmatrix}.  
\end{equation}

\vskip 1pc 

\begin{thm}\label{action exists} 
Let $p\geq 3$ be a prime number with $d=p-1$. Then we have the following:
\begin{enumerate}
\item[(1)] If a $d$-torus $\A_\Theta$ admits a canonical action of 
$\Z_p=\langle C_p\rangle$, namely $C_p^t\Theta C_p=\Theta$, then 
$\Theta$ must be of the form in (\ref{prime theta form}). 
\item[(2)] If $\Theta$ is a skew symmetric 
matrix of form in (\ref{prime theta form}), 
the group $\Z_p=\langle C_p\rangle$ canonically acts on $\A_\Theta$.
\item[(3)] If $\Theta$ is a skew symmetric 
matrix of form in (\ref{prime theta form}) such that 
the real numbers $1, \theta_0, \theta_1, \dots, \theta_{(p-3)/2}$ are 
independent over $\Z$, then the $d$-torus $\A_\Theta$  is simple. 
\end{enumerate}
\end{thm}
\begin{proof} 
For (1), first note that if $\Theta$ 
is a skew symmetric $d\times d$ matrix such that $\A_\Theta$ admits the canonical action of 
$\Z_p=\langle C_p\rangle$, then by Lemma~\ref{theta lemma}, 
$\Theta=(\theta_{ij})$ must be of the form in (\ref{general theta form}).  
Since the $p$th cyclotomic polynomial is $\Phi_p(x)=1+x+ \cdots+x^{p-1}$, 
with ${\bf a}=(-1,\dots,-1)^t$, 
one can repeat the computation performed in the proof of Lemma~\ref{theta lemma}
to obtain  that 
\begin{align*}
C_p^t\Theta C_p & \ = \ A^t\Theta A + A^t\Theta B  + B^t \Theta A +B^t\Theta B\\[3pt]
& \ = \  \begin{pmatrix} 
 \theta_{22} & \cdots & \theta_{2d} & 0 \\
 \vdots & \ddots  &\vdots & 0\\
 \theta_{d2} & \cdots & \theta_{d,d} & 0 \\
 0 &  \cdots & 0 & 0  
\end{pmatrix}
+
\begin{pmatrix} 
0 & \cdots & 0 &-\sum_j\theta_{2j}  \\
\vdots & \ddots &\vdots &\vdots \\
0 &  \cdots & 0 & -\sum_j\theta_{dj}  \\
0 &  \cdots & 0 & 0  
\end{pmatrix}\\[4pt]
& \ + 
\begin{pmatrix} 
0 & \cdots & 0 & 0 \\
\vdots & \ddots &\vdots &\vdots \\
0 &  \cdots & 0 & 0 \\
-\sum_i\theta_{i2} &  \cdots & -\sum_i\theta_{id}  & 0  
\end{pmatrix}
+
\begin{pmatrix} 
0 & \cdots & 0 & 0 \\
\vdots & \ddots &\vdots &\vdots \\
0 &  \cdots & 0 & 0 \\
0 &  \cdots & 0 & \sum_{i,j=1}^d \theta_{ij} 
\end{pmatrix}.
\end{align*}
But then from the fact that $C_p^t\Theta C_p$ is equal to the matrix   
$$\Theta=
\begin{pmatrix} 
 0 & \theta_{12} & \theta_{13} & \cdots & \theta_{1,p-2} & \theta_{1,p-1} \\[11pt]
 -\theta_{12} & 0 &\theta_{12} & \ \theta_{13}&\cdots  &\theta_{1,p-2}  \\[6pt]
 -\theta_{13} &  -\theta_{12} & 0 & \ \,\theta_{12} &  \cdots& \vdots\\[6pt]
 \vdots & \ddots & \ddots &\ddots &\ddots &  \theta_{13}  \\[7pt]
 -\theta_{1,p-2} & &   & & 0 &  \theta_{12} \\[9pt]
 -\theta_{1,p-1} & -\theta_{1,p-2}&   \cdots & -\theta_{13}  & -\theta_{12} & 0  
\end{pmatrix}, 
$$
comparing the last columns of $C_p^t\Theta C_p$ and $\Theta$, we have 
$$\theta_{1,p-1}=-\sum_j \theta_{2j},\ \theta_{1,p-2}=-\sum_j \theta_{3j},\  \dots, \
\theta_{13}=-\sum_j \theta_{d-1,j},\ \theta_{12}=-\sum_j \theta_{d,j}.$$
Note here that for any $\Theta$ of the above form, 
 $\sum_j\theta_{kj}+\sum_j\theta_{d-(k-1),j}=0$ for each $k=1, \dots, d/2$,  
which shows $$\theta_{1,p-1}=-\sum_j \theta_{2j}=\sum_j\theta_{d-1,j}=-\theta_{13}.$$ 
Similarly we see that $\theta_{1k}+ \theta_{1, d-(k-1)} =0$ 
for all $k=3, \dots, d/2$.

It is just a simple observation from the 
above computation  to see that $C_p^t\Theta C_p=\Theta$ 
holds for $\Theta$ in (\ref{prime theta form}), thus (2) follows. 
Also, it is not hard to see that 
if $1, \theta_0, \theta_1, \dots, \theta_{(p-3)/2}$ are independent over $\Z$, 
then the skew symmetric matrix $\Theta$ is nondegenerate,
which proves (3).  
\end{proof}

\vskip 1pc

\begin{cor}\label{cor for prime case} 
Let $p\geq 3$ be prime with $d=p-1$ and 
$\Theta$ be a nondegenerate skew symmetric $d\times d$ matrix 
of the form in (\ref{prime theta form}). 
Let $\af: \Z_p\to \Aut (\A_\Theta)$ be the canonical action of  
$\Z_p=\langle C_p\rangle$ on the simple $d$-torus $\A_\Theta$. 
Then the crossed product 
$\A_\Theta\rtimes_\af \Z_p$ is an AF algebra if and only if 
$p=3$ or $5$.
\end{cor}
\begin{proof}  By Proposition~\ref{AT}, 
we know that $\A_\Theta\rtimes_\af \Z_p$ is AF if and only if 
 $K_1( \A_\Theta\rtimes_\af \Z_p)=0$,
and this is the case when 
$p=3$ or $5$ from  Theorem~\ref{LL12}. 
\end{proof}

\vskip 1pc 

\section{Canonical actions on three dimensional simple tori}

\noindent
In the previous sections we considered the canonical action 
by the finite group $\Z_{n}$ on noncommutative 
$\phi(n)$-dimensional tori for $n\geq 3$. 
But $\phi(n)$ is always an even number for $n\geq 3$, 
thus the method using companion matrices 
won't work to find  finite groups acting on odd-dimensional tori.
  
In this section we will focus on 3-dimensional tori and 
show that no simple noncommutative $3$-tori   
admit the canonical actions of finite groups but the flip action:  
 Recall that the {\it flip action} on a $d$-torus $\A_\Theta$ 
 is the canonical action of $\Z_2=\{I_d, -I_d\}$ generated by the flip automorphism
 which sends each generator $l_\Theta(e_i)$ 
 to its adjoint $l_\Theta(e_i)^*=l_\Theta(-I_d\, e_i)$,  $i=1,\dots, d$.

\vskip 1pc

Two group actions  
$\alpha : G \rightarrow \Aut(\A)$ and $\beta :H \rightarrow \Aut(\BB)$ 
on $C^*$-algebras $\A$ and $\BB$ are 
{\it conjugate} if there exist a group isomorphism 
$\psi:G \rightarrow H$ and a $C^*$-isomorphism 
$\rho: \A \rightarrow \BB$ such that 
$\beta_{\psi(g)}\circ \rho = \rho \circ \alpha_g$ for all $g\in G$. 
This is an equivalence relation and two conjugate dynamical systems  
give rise to isomorphic crossed products.
In the following proposition we provide a sufficient condition on 
two matrices $\Theta,\Theta'\in \TT_d(\R)$ that     
every canonical action  on $\A_\Theta$ 
is conjugate to a canonical action  on $\A_{\Theta'}$. 
 
\vskip 1pc
 
\begin{prop}\label{conj} 
Let $\Theta$ and $\Theta'$ be two  matrices 
in $\TT_d(\R)$ such that 
$$\Theta  = (B^{-1})^t \Theta' B^{-1}$$ 
for some $B\in \GL_d(\Z)$. 
Then  we have the following:
\begin{enumerate}
\item[(1)] The map $\psi: G_{\Theta'} \to G_{\Theta}$,     
 $\psi(A) = BAB^{-1}$, is a group isomorphism.
\item[(2)] If $\alpha : G' \rightarrow \Aut(\A_{\Theta'})$ is a canonical action 
of a subgroup $G'$ of $G_{\Theta'}$, it is conjugate to 
the canonical action $\beta : \psi(G') \rightarrow \Aut(\A_{\Theta})$ of $\psi(G')$.
\end{enumerate}
\end{prop}

\begin{proof}
(1) Let $A\in G_{\Theta'}$. Then 
$\psi(A)=BAB^{-1}\in G_{\Theta}$ follows from  
\begin{align*}
((BAB^{-1})^{-1})^t \Theta (BAB^{-1})^{-1} 
&=(B^{-1})^t(A^{-1})^tB^t\Theta  BA^{-1}B^{-1}\\
   &= (B^{-1})^t(A^{-1})^t\Theta' A^{-1}B^{-1}  \\
   &= (B^{-1})^t\Theta' B^{-1}  \\
   &= \Theta.  
\end{align*}
Clearly $\psi$ is a group homomorphism with the inverse 
 $A \mapsto B^{-1}A B$, $A \in G_{\Theta}$. 

(2) One can check that  
$\rho:= \Ad U_B : \A_{\Theta'} \rightarrow \A_{\Theta}$ 
is an isomorphism  such that 
$\rho(l_{\Theta'}(x))=l_{\Theta}(Bx)$ for $x\in \Z^d$ (see Remark~\ref{isomorphisms}(1)),  
and then for $A\in G'$,   
$$\rho\circ \alpha_A(l_{\Theta'} (x)) =\rho (l_{\Theta'} (Ax)) = l_{\Theta }(BAx).$$
On the other hand, since $\bt$ is a canonical action, one also has
$$\beta_{\psi(A)}\circ \rho (l_{\Theta'}(x))  =\beta_{\psi(A)}(l_{\Theta}(Bx)) 
       =l_{\Theta}(\psi(A)Bx) =l_{\Theta}(BAx)$$
for all $A\in G'$ and $x\in \Z^d$, which completes the proof. 
\end{proof}

\vskip 1pc
  
As in the 2-dimensional case, 
if we have the complete list of group elements of finite order in $\GL_d(\Z)$ 
up to conjugacy, then by Proposition \ref{conj}  
we would be able to find all canonical actions 
by finite cyclic groups on $d$-dimensional tori. 
Actually this is possible for $d=3$ by virtue of the list (see Table \ref{table}) 
established in \cite{Th}.

\vskip .5pc  

\begin{center}
\begin{table}[h]
\caption{Elements of finite order in $\GL_3(\Z)$} \label{table}
\begin{tabular}{c|c}\hline
order & generators \\
\hline\hline
& \\
 & $A^2_1=\begin{pmatrix}{1}&{0}&{0}\\{0}&{-1}&{0}\\{0}&{0}&{-1}\end{pmatrix}$, 
    $A^2_2= \begin{pmatrix}{-1}&{0}&{0}\\{0}&{1}&{0}\\{0}&{0}&{1}\end{pmatrix}$,\\
& \\
  2& $A^2_3=\begin{pmatrix}{-1}&{0}&{0}\\{0}&{0}&{1}\\{0}&{1}&{0}\end{pmatrix}$,
    $A^2_4=\begin{pmatrix}{1}&{0}&{0}\\{0}&{0}&{-1}\\{}&{-1}&{0}\end{pmatrix}$, \\
& \\
  & $A^2_5= \begin{pmatrix}{-1}&{0}&{0}\\{0}&{-1}&{0}\\{0}&{0}&{-1}\end{pmatrix}$\\ 
 & \\  
\hline 
 & \\
3 & $A^3_1=\begin{pmatrix}{1}&{0}&{0}\\{0}&{0}&{-1}\\{0}&{1}&{-1}\end{pmatrix}$,
      $A^3_2=\begin{pmatrix}{0}&{1}&{0}\\{0}&{0}&{1}\\{1}&{0}&{0}\end{pmatrix}$ \\
&  \\
\hline
& \\
& $A^4_1=\begin{pmatrix}{1}&{0}&{0}\\{0}&{0}&{-1}\\{0}&{1}&{0}\end{pmatrix}$,
   $A^4_2= \begin{pmatrix}{-1}&{0}&{0}\\{0}&{0}&{1}\\{0}&{-1}&{0}\end{pmatrix}$, \\
4 & \\
 & $A^4_3=\begin{pmatrix}{1}&{0}&{1}\\{0}&{0}&{-1}\\{0}&{1}&{0}\end{pmatrix}$,
    $A^4_4= \begin{pmatrix}{-1}&{0}&{-1}\\{0}&{0}&{1}\\{0}&{-1}&{0}\end{pmatrix}$ \\
 & \\
\hline
& \\
& $A^6_1=\begin{pmatrix}{1}&{0}&{0}\\{0}&{0}&{-1}\\{0}&{1}&{1}\end{pmatrix}$,
   $A^6_2=\begin{pmatrix}{-1}&{0}&{0}\\{0}&{0}&{1}\\{0}&{-1}&{-1}\end{pmatrix},$ \\
6 & \\
 & $A^6_3=\begin{pmatrix}{-1}&{0}&{0}\\{0}&{0}&{1}\\{0}&{-1}&{1}\end{pmatrix}$,
    $A^6_4=\begin{pmatrix}{0}&{-1}&{0}\\{0}&{0}&{-1}\\{-1}&{0}&{0}\end{pmatrix}$\\
    & \\
\hline
\end{tabular}
\end{table}
\end{center}

\begin{thm}\label{3-torus}
The only canonical action by a nontrivial finite cyclic group  
on a simple $3$-dimensional torus is the flip action;
if $A\in \GL_3(\Z)$, $A\neq -I_3$, is a matrix in Table \ref{table},  
every $\Theta \in \TT_{3,A}(\R)$ is degenerate.
\end{thm}

\begin{proof}
If a 3-torus $\A_{\Theta'}$ admits a canonical action of a finite cyclic group 
$G'\subset G_{\Theta'}$, then $G'$ must be conjugate to a cyclic group generated by 
a matrix $A$ in the table, so that there exists $B\in \GL_3(\Z)$ such that 
$G'=\langle B^{-1}AB\rangle$ is generated by $B^{-1}AB$. 
 But then  by Proposition~\ref{conj},
 the cyclic group $\langle A\rangle$ canonically acts on 
 the 3-torus $\A_\Theta$, where $\Theta:= (B^{-1})^t \Theta' B^{-1}$. 
Also, it is rather obvious that $\Theta$ is nondegenerate exactly when 
$\Theta'$ is nondegenerate. 
Therefore  by Theorem~\ref{simple} it is enough to show that if $A$
is one of the matrices listed in the table and $A\neq -I_3(=A_5^2)$,  
every $\Theta\in \TT_{3,A}(\R)$ should be degenerate.   
 
Recall that $\Theta\in \mathcal{T}_d(\R)$ is degenerate  
if there exists a nonzero $x\in\Z^d$ such that 
$\exp(2\pi i \langle \Theta x , y \rangle )=1$ for all $y\in \Z^d$,   
or equivalently if there is a nonzero $x\in \Z^d$ with 
$\langle \Theta x, e_j \rangle \in \Z$ for all $j=1,\cdots,d$.   
Thus, to obtain the degeneracy of $\Theta\in \TT_{3,A}(\R)$, 
we  find  nonzero elements $x\in \Z^3$ with  $\Theta x\in \Z^3$. 
 
It is rather tedious to do the same calculation with all the 
matrices in the table, so here we only do with $A=A^2_1$ 
and leave the rest to readers.
If $\Theta=(\theta_{kj})\in  \TT_{3,A}(\R)$, that is  
$\displaystyle 
\Theta-A^t\Theta A=\begin{pmatrix}{0}&{2\theta_{12}}&{2\theta_{13}}\\
{-2\theta_{12}}&{0}&{0}\\{-2\theta_{13}}&{0}&{0}\end{pmatrix}$
is the zero matrix, 
then $\Theta$ must be of the form
$\displaystyle 
\begin{pmatrix}{0}& 0 & 0\\
 0 & 0 & s\\ 0 &{-s}& 0 \end{pmatrix}  
$ for an $s\in \R$. 
Any such matrix $\Theta$ is degenerate; in fact, 
$\Theta x=(0,0,0)^t\in\Z^3$ for any $x=(k,0,0)^t\in \Z^3$.
\end{proof} 

\vskip 1pc
 
\section{Canonical actions on four dimensional simple tori}  

\noindent
For the rest of the paper, 
we concretely examine examples of canonical actions on $4$-dimensional tori
$\A_\Theta$ by $\Z_n=\langle C_n\rangle$  with $\phi(n)=4$.

Since $\phi(n)=p_1^{r_1-1}(p_1-1)\cdots p_s^{r_s-1}(p_s-1)$ when  
$n=p_1^{r_1}\cdots p_s^{r_s}$ is the prime factorization of $n$,  
$\phi(n)=4$ implies that  $n$ should be equal to $5, 8, 10$ or $12$. 
For each of these $n$, since the $n$th cyclotomic polynomial is
$\Phi_5(x) =1+x+x^2+x^3+x^4$, 
$\Phi_8(x) =1+x^4$, 
$\Phi_{10}(x) =1-x+x^2-x^3+x^4$, and
$\Phi_{12}(x) =1-x^2+x^4$, respectively, 
the companion matrix $C_n$ (of order $n$) of (\ref{comp matrix}) 
is given by: 
\begin{equation}\label{C_n}
\begin{aligned}
C_5&=   \begin{pmatrix}
       {0}&{0}&{0}&{-1}\\{1}&{0}&{0}&{-1}\\{0}&{1}&{0}&{-1}\\{0}&{0}&{1}&{-1}
        \end{pmatrix}, \ \ 
C_8=   \begin{pmatrix} 
       {0}&{0}&{0}&{-1}\\{1}&{0}&{0}&{0}\\{0}&{1}&{0}&{0}\\{0}&{0}&{1}&{0}
       \end{pmatrix},  \\ 
C_{10}&= \begin{pmatrix} 
       {0}&{0}&{0}&{-1}\\{1}&{0}&{0}&{1}\\{0}&{1}&{0}&{-1}\\{0}&{0}&{1}&{1}
       \end{pmatrix}, \ \  
C_{12}= \begin{pmatrix}
       {0}&{0}&{0}&{-1}\\{1}&{0}&{0}&{0}\\{0}&{1}&{0}&{1}\\{0}&{0}&{1}&{0}
       \end{pmatrix}.
\end{aligned}
\end{equation}

Recall that the group $\Z_n=\langle C_n\rangle$  canonically acts on 
a $4$-torus $\A_\Theta$ exactly when $C_n^t \Theta C_n=\Theta$, 
or equivalently when $\Theta\in \TT_{4,C_n}(\R)$.
Since every skew symmetric matrix $\Theta\in  \TT_{4,C_n}(\R)$ 
has the form in (\ref{general theta form}), by a simple computation 
we see that 
\begin{eqnarray}
\TT_{4,C_n}(\R) =\Big\{
\begin{pmatrix} 
{0}&{\theta}&{\mu}&{\nu_n}\\
 &{0}&{\theta}&{\mu}\\
 & &{0}&{\theta}\\
 & & &{0}
\end{pmatrix} : \theta, \mu \in \R \Big\}, \label{T_n}
\end{eqnarray}
where $\nu_5=-\mu$, $\nu_8=\theta$, $\nu_{10}=\mu$, and $\nu_{12}=2\theta$. 

Moreover,  $\Theta\in \TT_{4,C_n}(\R)$ is easily seen to be nondegenerate 
whenever $1$, $\theta$, $\mu$ are independent over $\Z$.
 
\vskip 1pc
 
\begin{prop} If $\Theta$ is a skew symmetric 
$4\times 4$ matrix in $\TT_{4,C_{n}}(\R)$ for some $n=5,8,10, 12$, 
then the noncommutative $4$-torus $\A_\Theta$ admits the canonical action $\af$
by the finite group $\Z_n=\langle C_n\rangle$ generated by 
$C_n$ in {\rm (\ref{C_n})}. 
If $\Theta\in \TT_{4,C_{n}}(\R)$ is nondegenerate, the crossed product 
$\A_\Theta\rtimes_\af \Z_n$ is an AF algebra.
\end{prop}
\begin{proof} For $n$ even, $K_1(\Z^4\rtimes_\af \Z_n)=0$ by 
Theorem~\ref{LL12} and $\A_\Theta\rtimes_\af \Z_n$ is AF 
by Proposition~\ref{AT}. 
$\A_\Theta\rtimes_\af \Z_5$ is AF by 
Corollary~\ref{cor for prime case}.
\end{proof}

\vskip 1pc

\begin{remark} Let $m=p_1^{k_1}p_2^{k_2}\cdots p_t^{k_2}$ 
be the prime factorization of an integer $m\in \mathbb N$, where  
$p_1<p_2<\cdots < p_t$ are primes. 
Then it is known (\cite[Theorem 2.7]{KP}) that the group 
$\GL_n(\Z)$ has an element of order $m$ if and only if 
\begin{enumerate}
\item[(1)] $\sum_{i=1}^t (p_i-1)p_i^{k_i -1} -1\leq n$ for $p_1^{k_1}=2$, or
\item[(2)] $\sum_{i=1}^t (p_i-1) p_i^{k_1 -1}\leq n$ otherwise.
\end{enumerate}
Thus, with $n=4$, we see that any possible finite order of a matrix in 
$\GL_4(\Z)\setminus\{I_4\}$ is  
one of   $ 2,3,4,5,6,8,10, 12$. 
It should be noted that the action by $\Z_5$    
  is not conjugate to any product action (on $\A_\theta\otimes \A_\theta$)  
  of two canonical actions by finite cyclic subgroups $F(\subset \SL_2(\Z))$ on $\A_\theta$
  because $F$ is necessarily isomorphic to $\Z_2, \Z_3, \Z_4$, or $\Z_6$. 
  The actions by these $F$ are the only finite group actions on noncommutative tori 
  found in the literature at least to the knowledge of the authors, which led us to 
  work on finite group actions on higher dimensional tori. 
\end{remark}

\vskip 1pc

Now we consider 4-dimensional noncommutative tori that are isomorphic to 
the tensor product $\A_\theta\otimes \A_\theta$ of 
an irrational rotation algebra $\A_\theta$ with itself. 
If  $\A_\Theta$ is associated with  
the following skew symmetric matrix   
\begin{equation}\label{theta-theta}
\Theta=
\begin{pmatrix}
{0}&{\theta}&{0}&{0}\\
{-\theta}&{0}&{0}&{0}\\
{0}&{0}&{0}&{\theta}\\
{0}&{0}&{-\theta}&{0}
\end{pmatrix}, 
\end{equation}
it is easily seen to be isomorphic to the tensor product 
$\A_\theta \otimes \A_\theta$.
Moreover the following skew symmetric matrices $\Theta_{n,\theta}$,
\begin{equation}\label{Theta_5,10}
\Theta_{5,\theta}=\Theta_{10,\theta} =
\begin{pmatrix}
    {0}&{\theta}&{0}&{0}\\
    {-\theta}&{0}&{\theta}&{0}\\
    {0}&{-\theta}&{0}&{\theta}\\
    {0}&{0}&{-\theta}&{0}
\end{pmatrix}\in \TT_{4,C_5}(\R)\cap \TT_{4,C_{10}}(\R),
\end{equation}
\begin{equation}\label{Theta_8,12}
\Theta_{8,\theta}=\Theta_{12,\theta} =
\begin{pmatrix}
{0}&{0}&{\theta}&{0}\\
{0}&{0}&{0}&{\theta}\\
{-\theta}&{0}&{0}&{0}\\
{0}&{-\theta}&{0}&{0}
\end{pmatrix}\in \TT_{4,C_8}(\R)\cap \TT_{4,C_{12}}(\R),
\end{equation}
(see (\ref{T_n})) give rise to $4$-dimensional tori 
isomorphisc to $\A_\theta\otimes \A_\theta$:

\vskip 1pc

\begin{lemma} \label{B_n}
Let $\Theta$ be the matrix in {\rm (\ref{theta-theta})} 
with  $\theta\in \R$ and $\Theta_{n,\theta}$ be one of the 
matrices in  {\rm (\ref{Theta_5,10})} or {\rm (\ref{Theta_8,12})} 
for $n=5, 8, 10, 12$. 
We then have the following: 
\begin{enumerate}
\item[(1)] There exists $B_n\in \GL_n(\Z)$ with 
  $B_n^t\Theta_{n,\theta} B_n = \Theta$.
\item[(2)] $G_\Theta = B_n^{-1} G_{\Theta_{n,\theta}} B_n$. 
\item[(3)]  $\A_{\Theta_{n,\theta}}$ is isomorphic to $\A_\Theta$. 
\end{enumerate} 
\end{lemma}
\begin{proof}
(1) The following  $B_n$ is the desired matrix for each $n$:
\[
B_5=B_{10}:=
\begin{pmatrix}
   {1}&{0}&{0}&{0}\\
   {0}&{1}&{0}&{0}\\
   {0}&{0}&{1}&{0}\\
   {0}&{1}&{0}&{1}
\end{pmatrix},~
B_8=B_{12}:=\begin{pmatrix}
   {1}&{0}&{0}&{0}\\
   {0}&{0}&{1}&{0}\\
   {0}&{1}&{0}&{0}\\
   {0}&{0}&{0}&{1}
\end{pmatrix}. 
\]

(2) and (3) then follow from Proposition~\ref{conj} and its proof.
\end{proof}

\vskip 1pc
\noindent
 Since, in the above situation,  
 $C_n\in G_{\Theta_{n,\theta}}$, for $\theta \in \R$ and $n=5,8,10,12$, and    
 $$C\mapsto (B_n)^{-1}C  B_n: G_{\Theta_{n,\theta}}\to G_\Theta$$ 
 is a group isomorphism,  
 we see that  
 $$A_n:=(B_n)^{-1}C_n B_n $$  
 are the matrices (acting on $\A_\Theta$) 
 of order $n$ for $n=5,8,10,12$ by Lemma~\ref{B_n}(ii), and actually given by
\begin{equation}\label{A_n}
\begin{aligned}
A_5&=
\begin{pmatrix}
{0}&{-1}&{0}&{-1}\\{1}&{-1}&{0}&{-1}\\{0}&{0}&{0}&{-1}\\{-1}&{0}&{1}&{0}
\end{pmatrix},~
A_8=
\begin{pmatrix}
{0}&{0}&{0}&{-1}\\{0}&{0}&{1}&{0}\\{1}&{0}&{0}&{0}\\{0}&{1}&{0}&{0}
\end{pmatrix},\\
A_{10}&=
\begin{pmatrix}
{0}&{-1}&{0}&{-1}\\
{1}&{1}&{0}&{1}\\
{0}&{0}&{0}&{-1}\\
{-1}&{0}&{1}&{0}
\end{pmatrix},~
A_{12}=
\begin{pmatrix}
{0}&{0}&{0}&{-1}\\
{0}&{0}&{1}&{1}\\
{1}&{0}&{0}&{0}\\
{0}&{1}&{0}&{0}
\end{pmatrix}. 
\end{aligned}
\end{equation}

\vskip 1pc

\begin{prop}\label{4-torus} 
Every  4-dimensional noncommutative torus of the form 
$\A_\theta\otimes \A_\theta$ admits 
a canonical action $\af_n$ of $\Z_n=\langle A_n\rangle$ 
in {\rm (\ref{A_n})} for $n=5,8,10,12$. 
More precisely, 
if $\A_\theta\otimes \A_\theta=C^*(u_1,u_2)\otimes C^*(u_3,u_4)$ is 
generated by 4 unitaries $u_i$'s satisfying 
$u_2 u_1=\exp(2\pi i\theta)u_1 u_2$, $u_4 u_3=\exp(2\pi i\theta)u_3 u_4$, 
and $u_k u_l= u_l u_k$ for $k=1,2$ and $l=3,4$,  
we have the following: 
\begin{enumerate} 
\item[(1)] $\af_5:\Z_5\to \Aut(\A_\theta\otimes\A_\theta)$ is induced by the 
automorphism;   
\begin{align}\label{case5}
u_1 & \mapsto u_2u_4^*,\  u_2\mapsto \exp(\pi i \theta)u_1^*u_2^*, \\
u_3 & \mapsto  u_4, \ \ \ \  u_4 \mapsto  \exp(\pi i \theta)u_1^*u_2^*u_3^*. \nonumber
\end{align}

\item[(2)] $\af_8:\Z_8\to \Aut(\A_\theta\otimes\A_\theta)$ is induced by the 
automorphism;       
\begin{align}
 u_1 & \mapsto u_3,\ \ \ u_2 \mapsto u_4,\\
u_3& \mapsto u_2,\ \ \ u_4 \mapsto u_1^*.\nonumber
\end{align}
 
\item[(3)] $\af_{10}:\Z_{10}\to \Aut(\A_\theta\otimes\A_\theta)$ is induced by the 
automorphism;       
\begin{align}
u_1 & \mapsto u_2u_4^*,~u_2 \mapsto \exp(-\pi i \theta)u_1^*u_2,\\
u_3 & \mapsto u_4,\ \ \ \ u_4 \mapsto \exp(-\pi i \theta)u_1^*u_2u_3^*.\nonumber                                                                            
\end{align}

\item[(4)]  $\af_{12}:\Z_{12}\to \Aut(\A_\theta\otimes\A_\theta)$ is induced by the 
automorphism;       
\begin{align}
u_1& \mapsto u_3,~u_2 \mapsto u_4,\\
u_3& \mapsto u_2,~u_4 \mapsto \exp(-\pi i \theta)u_1^*u_2.\nonumber
\end{align}

\end{enumerate}
\end{prop}   

\begin{proof} 
We only show the case (1) here because the rest can be 
done similarly. 
Since $\A_\Theta$ is the twisted group algebra 
$C^*(\Z^4,\omega_\Theta)=C^*\{l_\Theta (e_i): 1\leq i\leq 4\}$ 
(\ref{nc torus})with the identification 
$u_i:=l_\Theta(e_i)$ for each $i$, the action $\af_5$ of 
$\Z_5=\langle A_5\rangle$ is determined by $\af_5(k)(l_\Theta(e_i))$  
for $k\in \Z_5$ and $ 1\leq i\leq 4$.
 For convenience, we write simply $\af$ for $\af_5(1)$ to have:
\begin{align*} 
\af(l_\Theta(e_1))& =l_\Theta (A_5 e_1) =l_\Theta (e_2-e_4)\\
                  & = \overline{\omega_\Theta(e_2,-e_4)}\, l_\Theta (e_2) l_\Theta(e_4)^*\\
                  & = \exp(-\pi i \langle \Theta e_2,-e_4\rangle)\, l_\Theta (e_2) l_\Theta(e_4)^*\\
                  & = l_\Theta (e_2) l_\Theta(e_4)^*, \\
\af(l_\Theta(e_2))& = l_\Theta (A_5 e_2)= \exp(\pi i \theta)
                      l_\Theta (e_1)^*l_\Theta (e_2)^*,\\
\af(l_\Theta(e_3))& = l_\Theta (A_5 e_3)= l_\Theta (e_4),\\    
\af(l_\Theta(e_4))& = l_\Theta (A_5 e_4)= \exp(\pi i \theta)
                      l_\Theta (e_1)^*l_\Theta (e_2)^*l_\Theta (e_3)^*.                                   
\end{align*}                  
Thus $\af_5$ is the action of $\Z_5$ generated by the automorphism $\af$ sending 
$u_1$ to $u_2 u_4^*$ and so on as stated in (\ref{case5}). 
\end{proof}

\end{document}